\documentclass[10pt]{amsart}

\usepackage{amsmath}
\usepackage{amsfonts}
\usepackage{amsthm}
\usepackage{amssymb}

\usepackage{mathrsfs}

\usepackage{hyperref}

\usepackage{xypic}
\usepackage[active]{srcltx}
\usepackage{graphicx}
\usepackage{color}

\usepackage{enumitem}

\newcommand{\pv}[1]{\ensuremath{{\mathsf{#1}}}}
\newcommand{\F}[1]{\ensuremath\widehat{#1^\ast}}
\newcommand{\FS}[1]{\ensuremath\widehat{#1^+}}

\newcommand{\FV}[2]{\ensuremath\widehat{F}_{\pv {#1}}(#2)}

\newcommand{\J}{\ensuremath{\mathcal{J}}}
\newcommand{\R}{\ensuremath{\mathcal{R}}}
\renewcommand{\L}{\ensuremath{\mathcal{L}}}
\renewcommand{\H}{\ensuremath{\mathcal{H}}}
\newcommand{\K}{\ensuremath{\mathcal{K}}}

\newcommand{\NN}{\mathbb N}
\newcommand{\ZZ}{\mathbb Z}

\newcommand{\ci}[1]{\ensuremath{{\mathcal {#1}}}}

\newcommand{\Fin}{\ensuremath{{\operatorname{Fin}}}}
\newcommand{\Pref}{\ensuremath{{\operatorname{Pref}}}}

\newcommand{\Block}{\ensuremath{{\operatorname{\mathcal B}}}}

\newcommand{\Cl}[1]{\ensuremath{\ci #1}}

\newcommand\malcev{\mathop{\raise1pt\hbox{\footnotesize$\bigcirc$\kern-8pt\raise1pt
      \hbox{\tiny$m$}\kern1pt}}}

\newtheorem{Thm}{Theorem}[section]

\newtheorem{Prop}[Thm]{Proposition}
\newtheorem{Lemma}[Thm]{Lemma}
\newtheorem{Cor}[Thm]{Corollary}

{\theoremstyle{definition}

\newtheorem{Example}[Thm]{Example}
\newtheorem{Rmk}[Thm]{Remark}
}

\numberwithin{equation}{section}

\begin{document}

\title[A profinite approach to complete bifix decodings]{A profinite approach to complete bifix decodings of recurrent languages}
\thanks{Work carried out in part at City College of New York, CUNY, whose hospitality is gratefully acknowledged, with the support of the FCT sabbatical scholarship SFRH/BSAB/150401/2019,
  and it was partially supported by the Centre for Mathematics of the University of Coimbra - UIDB/00324/2020, funded by the Portuguese Government through FCT/MCTES}

\author{Alfredo Costa}

\address{University of Coimbra, CMUC, Department of Mathematics,
  Apartado 3008, EC Santa Cruz,
  3001-501 Coimbra, Portugal.}
\email{amgc@mat.uc.pt}

\subjclass[2020]{20M05, 68Q70, 94B35, 20E18, 37B10}

\keywords{Code, complete bifix decoding, free profinite monoid, profinite group, symbolic dynamics, subshift, eventual conjugacy}

\begin{abstract}  
  We approach the study of complete bifix decodings of (uniformly) recurrent languages with
  the help of the free profinite monoid. We show that the complete bifix decoding of a uniformly recurrent language $F$
  by an $F$-charged rational complete bifix code is uniformly recurrent. An analogous result is obtained for recurrent languages.
  As an application, we show that the Sch\"utzenberger group of an irreducible symbolic dynamical system is an invariant of eventual conjugacy.
\end{abstract}

\maketitle

%\tableofcontents

\section{Introduction}

In this paper we see symbolic dynamical systems, codes, and free profinite monoids in interplay.
The following introductory paragraphs provide a first picture of how these fields come together in this work. Several details are deferred to later sections or to the bibliography.
We indicate the books~\cite{Lind&Marcus:1996,Fogg:2002} as references for symbolic dynamics; the book~\cite{Berstel&Perrin&Reutenauer:2010} for codes;
and the book~\cite{Almeida&ACosta&Kyriakoglou&Perrin:2020b} for an introduction to free profinite monoids and their connections with codes and symbolic dynamics.

When dealing with a subshift $\Cl S\subseteq A^{\ZZ}$, we are often led 
to decode $\Cl S$ into a new symbolic dynamical system $\Cl S'\subseteq X^\ZZ$.
This decoding is quite often better understood in terms of sets of finite words.
From that viewpoint, we decode the language $F\subseteq A^*$  of blocks of elements of $\Cl S$ into
a language $F'\subseteq X^*$.
An important example is the following: for a positive integer $n$,
and for $Z=A^n$, let $F'=F\cap Z^*$; then $F'$ is the language
of blocks of a subshift, called the \emph{$n$-th higher power} of~$\Cl S$, over the finite alphabet $X=F\cap Z$.
Two subshifts are said to be \emph{eventually conjugate} if for all sufficiently large $n$ their $n$-th higher powers are conjugate.
This notion received great attention since introduced by Williams, who conjectured that for finite type subshifts eventual conjugacy coincides with conjugacy~\cite{Williams:1973}. This conjecture was later disproved by Kim and Roush~\cite{Kim&Roush:1992,Kim&Roush:1999}.

We now look at the following key notion in the algebraic theory of codes:
a code~$Z$ over the alphabet $A$ is \emph{complete} if every element of $A^*$ is the prefix of some element of $Z^*$ and the suffix of some element of $Z^*$.
The set $Z=A^n$ is an example of a complete bifix code.
The 2012 seminal paper~\cite{Berstel&Felice&Perrin&Reutenauer&Rindone:2012} prompted a research line giving center stage to the following operation: take a recurrent subset~$F$ of~$A^*$ (that is, $F$ is the language of blocks of an irreducible subshift of $A^\ZZ$), and a
complete bifix code $Z$, and
investigate the properties of the intersection $X=F\cap Z$.
For example, in~\cite{Berstel&Felice&Perrin&Reutenauer&Rindone:2012} it is shown
that $X$ is finite if $F$ is uniformly recurrent (that is, if $F$ is the language of blocks of a minimal subshift of $A^\ZZ$).
Since $Z$ is a code, an element of $F\cap Z^*$ can be seen uniquely as an element of the free monoid over the finite alphabet $X=F\cap Z$,
thus making $F\cap Z^*$ the language of a symbolic dynamical system over the alphabet~$X$.
The language $F\cap Z^*$ is then said to be a \emph{complete bifix decoding} of $F$. In general, the hypothesis that
$F$ is uniformly recurrent does not imply that $F\cap Z^*$ is recurrent (over the alphabet $X$), even in the case where $Z=A^n$.
One of the main results of~\cite{Berthe&Felice&Dolce&Leroy&Perrin&Reutenauer&Rindone:2015d} states that if $F$ is a uniformly recurrent dendric language (a class of languages
that includes the famous Sturmian languages and whose precise definition we recall in Section~\ref{sec:codes}), then every complete bifix decoding of $F$ is also a uniformly recurrent dendric language. It turns out that
 in the proof of that result, one of the most ``hairy points'' (to borrow the words used by Perrin in his DLT 2018 survey~\cite{Perrin:2018})
 is showing that the complete bifix decoding of a uniformly recurrent dendric language is indeed uniformly recurrent.

 The application of the free profinite monoid within the line of research inaugurated by Berstel et al.~in the aforementioned paper~\cite{Berstel&Felice&Perrin&Reutenauer&Rindone:2012} was first made in~\cite{Kyriakoglou&Perrin:2017,Almeida&ACosta&Kyriakoglou&Perrin:2020}, in the study of the group of a complete bifix code.
 That ``profinite'' approach is centered around the notion of \emph{$F$-charged code}, a notion whose precise definition involves a maximal subgroup of the free profinite monoid depending on the recurrent language $F$. In this paper we use that approach
 to significantly extend the property that all complete bifix decodings of uniformly recurrent dendric languages are also uniformly recurrent (the aforesaid ``hairy point'', for which this paper gives therefore a new proof). More concretely: we show that the decoding of any uniformly recurrent
language $F$ by an $F$-charged rational complete bifix code is always uniformly recurrent (Theorem~\ref{t:charged-decoding-uniformly-recurrent}).
A similar general result concerning the preservation of recurrence is obtained in Theorem~\ref{t:charged-decoding-recurrent}.

We now say some more words about the maximal subgroup of the free profinite monoid briefly mentioned in the previous paragraph.
Let $\Cl S$ be an irreducible subshift of $A^\ZZ$. The \emph{Sch\"utzenberger group of $\Cl S$}, denoted $G(\Cl S)$,
is a profinite group, introduced by Almeida in~\cite{Almeida:2003cshort},
naturally located as a maximal subgroup of the free profinite monoid over $A$, with the elements of $G(\Cl S)$ being certain limits of blocks of elements of $\Cl S$.
The invariance of $G(\Cl S)$ under conjugacy was first proved in~\cite{ACosta:2006}.
As an application of our main result about complete bifix decodings of recurrent languages, applied to the special case of bifix codes of the form $Z=A^n$, we show that $G(\Cl S)$ is an invariant of eventual conjugacy (Theorem~\ref{t:invariance-under-eventual-conjugacy}).  In~\cite{ACosta&Steinberg:2021} it is shown that
 $G(\Cl S)$ is invariant under \emph{flow equivalence}, a relation
 coarser than conjugacy.
It's known that eventual conjugacy of irreducible subshifts of finite type
implies flow equivalence (this is a corollary of the classification of irreducible subshifts of finite type up to flow equivalence~\cite{Franks:1984,Parry&Sullivan:1975,Bowen&Franks:1977},
as the reader may check when reading the book of Lind and Marcus~\cite{Lind&Marcus:1996}, namely its Sections 3.4 and 13.6, or when reading Boyle's expository article~\cite{Boyle:2002}). It is an open
question as to whether this implication holds for all subshifts of finite type.\footnote{The implication for general subshifts of finite type was stated without proof as a "difficult result" in Boyle's 1997
Temuco lectures~\cite[Section 3.7.9]{Boyle:1997}. The author withdraws the claim now~\cite{Boyle:pc}, as he has not been able
to recover or produce a proof.}
For the whole class of irreducible subshifts, it is unclear when the implication holds.

  Since their introduction, the groups of the form
$G(\Cl S)$ contributed significantly to the study of the structure of free profinite monoids.
In~\cite{Almeida&Volkov:2006} it is shown that $G(\Cl S)$ is a free procyclic group if $\Cl S$ is periodic.
The main result of~\cite{ACosta&Steinberg:2011} states
 that if $\Cl S$ is a nonperiodic irreducible sofic subshift, then $G(\Cl S)$ is a free profinite group 
 with rank $\aleph_0$; interestingly, the invariance under conjugacy of $G(\Cl S)$ was used in the proof. The case of minimal nonperiodic subshifts has a much more rich landscape, far from being fully understood.
 When $\Cl S$ is minimal, then $G(\Cl S)$ may be a free profinite group or not. For example, $G(\Cl S)$ is a free profinite group of rank the size of the alphabet
 when $\Cl S$ is dendric~\cite{Almeida&ACosta:2016b}, but if $\Cl S$ is the Th\"ue-Morse subshift then $G(\Cl S)$ is not free~\cite{Almeida&ACosta:2013}.
 See~\cite{Goulet-Ouellet:2021,Goulet-Ouellet:2022c} for criteria helping to decide if $G(\Cl S)$ is free or not, when $\Cl S$ is defined by a primitive substitution $\varphi$.
 These criteria are associated to finite algebraic invariants of $G(\Cl S)$ that may be easily computed from the matrix associated to $\varphi$ (cf.~comments at the end of \cite[Section 4.4]{Goulet-Ouellet:2022c}). We now know that these
 algebraic invariants are invariants of eventual conjugacy of the dynamical system $\Cl S$.

 Finally, we mention that our proof of the invariance under eventual conjugacy of $G(\Cl S)$
 allows, with minimal adaptations, a generalization involving relatively free profinite monoids (Theorem~\ref{t:relatively-free-invariance-under-eventual-conjugacy}).

\section{Subshifts, free profinite monoids, and connections between them}

Along this paper, all alphabets are assumed to be nonempty finite sets.

\subsection{Subshifts}

Consider an alphabet $A$. A subset $F$ of the free monoid $A^*$
is said to be \emph{factorial} if it is closed under taking factors,
and is said to be \emph{prolongable} if for all $u\in F$ there are $a,b\in A$
such that $au$ and $ub$ belong to~$F$. 

Let $x=(x_i)_{i\in\ZZ}\in A^{\ZZ}$. A word
of the form $x_{i}\cdots x_{i+n}$, with $i\in\ZZ$ and~$n\in\NN$,
is a \emph{block} of $x$. We also allow the empty word
to be a block of~$x$. We denote by $\Block(x)$ the set of blocks of~$x$.
Consider in $A^{\ZZ}$ the product topology, with $A$ having
the discrete topology. Since we are assuming $A$ finite,
the space $A^{\ZZ}$ is compact, by Tychonoff's theorem. (The Hausdorff property is being included in the definition of compact space.)
Consider the shift map $\sigma_A\colon A^{\ZZ}\to A^{\ZZ}$, given
by $\sigma_A((x_i)_{i\in\ZZ})=(x_{i+1})_{i\in\ZZ}$.
A \emph{symbolic dynamical system}, or \emph{subshift}, of $A^{\ZZ}$ is a nonempty closed subspace $\Cl S$ of~$A^{\ZZ}$
such that~$\sigma_A(\Cl S)=\Cl S$.
Given $F\subseteq A^*$, let $\Cl S_F=\{x\in A^\ZZ\mid \Block(x)\subseteq F\}$.
The mapping $F\mapsto\mathcal S_F$
is a bijection from the set of nonempty factorial prolongable
languages of $A^*$ to the set of subshifts of~$A^\ZZ$,
whose inverse maps a subshift $\Cl S\subseteq A^\ZZ$
to the language $\Block(\Cl S)=\bigcup_{x\in \Cl S}\Block(x)$.
In view of this, we say that a nonempty factorial prolongable
language  is a \emph{subshift language}. 

Let $F$ be a subshift language of $A^*$.
One says that $F$ is \emph{recurrent} (over the alphabet $A$)
when for every $u,v\in F$ there is $w\in F$ such that $uwv\in F$.
Also, $F$ is \emph{uniformly recurrent} (over $A$) when, for every $u\in F$,
there is an integer $N$ such that $u$ is a factor of every
word of $F$ with length at least $N$. Every uniformly recurrent language
is recurrent. These definitions have dynamical meaning:
a subshift $\Cl S$ has a dense positive forward orbit (one says that such a subshift is \emph{irreducible})
if and only if $\Block(\Cl S)$ is recurrent;
and $\Cl S$ is a minimal subshift for the inclusion (equivalently, every orbit in $\Cl S$ is dense in~$\Cl S$) if and only if $\Block(\Cl S)$ is uniformly recurrent.

We refer to the books~\cite{Lind&Marcus:1996,Fogg:2002} for information about
subshifts. The first book is more focused
on subshifts whose language of blocks is rational, the \emph{sofic} subshifts (which include the \emph{finite type subshifts} already mentioned at the introduction of this paper),
while in the second book we find extensive information
about a special class of minimal subshifts, the \emph{primitive substitutive} subshifts.
The two books complement each other, as minimal subshifts and irreducible sofic subshifts have in common only the periodic subshifts.
Primitive substitutive subshifts may be defined as follows.
Let $\varphi$ be an endomorphism of~$A^*$, where the finite set~$A$ has at least two elements.
  If there is a positive integer $n$ such that
  for every $a,b\in A$, the letter $b$ is a factor of $\varphi^n(a)$,
  then $\varphi$ is called a \emph{primitive substitution}; in that case the set $F_\varphi$ of factors of words of the form
  $\varphi^k(a)$, with $k\in\NN$ and $a\in A$,
  is indeed a uniformly recurrent language.

  \subsection{Profinite monoids and their Green's relations}

  For books about profinite monoids see e.g.~\cite{Rhodes&Steinberg:2009qt}, \cite{Almeida&ACosta&Kyriakoglou&Perrin:2020b}, and~\cite{Ribes&Zalesskii:2010} for profinite groups.
  For the sake of fixing terminology and notation, we recall
some basic definitions and results. We begin with the Green quasi-orders in a monoid~$M$: for every $x,y\in M$, we have
\begin{itemize}
\item $x\leq_{\J_M}y$ if and only if $y$ is a factor of $x$ (that is, if and only if $x\in MyM$);
\item $x\leq_{\R_M}y$ if and only if $y$ is a prefix of $x$ (that is, if and only if $x\in yM$);
\item $x\leq_{\L_M}y$ if and only if $y$ is a suffix of $x$ (that is, if and only if $x\in My$).
\end{itemize}
The corresponding induced Green's equivalence relations are $\J_M$, $\R_M$ and $\L_M$. We also consider the intersection
$\H_M=\R_M\cap\L_M$. The subscript $M$ in $\K_M$ may be dropped, for $\K\in\{\J,\R,\L,\H\}$. 
An $\H_M$-class $H$ contains an idempotent if and only if it is a subgroup of $M$ (that is, a subsemigroup of $M$ with a group structure),
and in fact every subgroup of $M$ is contained in an $\H_M$-class, for which reason such $\H_M$-classes are referred to as the \emph{maximal subgroups} of $M$.

A \emph{compact monoid} is a monoid endowed with a compact topology for which
the multiplication is continuous. If $M$ is a compact monoid, then $M$ is a \emph{stable monoid}, meaning
that $\leq_{\R_M}\cap\, \J_M=\R_M$ and $\leq_{\L_M}\cap\, \J_M=\L_M$.
In a compact monoid $M$, all Green's relations are closed, and so in particular maximal subgroups are closed subgroups.

A \emph{profinite monoid} is an inverse limit of finite monoids in the category of compact monoids (we view finite monoids as
compact monoids).  All closed subgroups of a profinite monoid are profinite groups, that is, inverse limits of finite groups.   

Let $M$ be a profinite monoid. For every $x\in M$ and $k\in\ZZ$, the sequence $(x^{n!+k})_{n\geq |k|}$ converges to an element of $M$ denoted $x^{\omega+k}$, and taking $k=0$
we get the idempotent $x^\omega=\lim x^{n!}$.
An element $x$ in a monoid $M$ is \emph{regular} if $x\in xMx$.  For $\K\in\{\J,\R,\L\}$, a $\K_M$-class contains a regular element if and only if it contains only regular elements, if and only if it contains some idempotent. A $\K_M$-class with regular elements is called a \emph{regular $\K_M$-class}.
All maximal subgroups contained in a regular $\J_M$-class $J$ are isomorphic profinite groups, and so they can be identified as the same profinite group, referred to as the \emph{Sch\"utzenberger group} of $J$.

\subsection{Free profinite monoids}
In the following lines, we mostly follow the notation and approach of~\cite{Almeida&ACosta&Kyriakoglou&Perrin:2020b}.

If $u,v$ are distinct elements of the free monoid $A^*$,
then there is some homomorphism $\varphi\colon A^*\to N$
onto a finite monoid $N$ satisfying $\varphi(u)\neq\varphi(v)$.
Let $r(u,v)$ be the smallest possible cardinal for such a monoid $N$.
Consider the completion $\F A$ of $A^*$ under the
metric~$d$ defined, when $u\neq v$, by $d(u,v)=2^{-r(u,v)}$.
The free monoid~$A^*$ embeds as a topological subspace of~$\F A$.
Moreover, the monoid structure of~$A^*$
extends uniquely to a monoid structure on $\F A$,
making $\F A$ a compact monoid. In fact,
$\F A$ is a profinite monoid, known as the
\emph{free profinite monoid generated by~$A$} because of the following universal property: for every
map $\varphi\colon A\to M$
into a profinite monoid $M$, there is
a unique continuous homomorphism $\hat\varphi\colon \F A\to M$
extending~$\varphi$.

We view the elements of $\F A$
as generalizations of words, for which reason each element
of $\F A$ is called a \emph{pseudoword} over $A$.
The elements of $A^*$ are isolated points in $\F A$,
and they are the only isolated points there, as $A^*$ is dense in $\F A$.
Each element of $\F A\setminus A^*$ is
said to be an \emph{infinite pseudoword},
while each word in $A^*$ is a \emph{finite pseudoword}.
The set $\F A\setminus A^*$ is an ideal of
$\F A$.

We may also consider profinite semigroups, \emph{mutatis mutandis}.
It turns out that the \emph{free profinite semigroup} generated by $A$, denoted $\FS A$, is (isomorphic) to $\F A\setminus\{\varepsilon\}$, where
$\varepsilon$ is the empty word.

The next theorem gives a hint for why free profinite monoids are important.

 \begin{Thm}[{cf.~\cite[Theorem 3.6.1]{Almeida:1994a}}]\label{t:rational-open}
   A language $L\subseteq A^*$
   is rational if and only if its topological closure $\overline{L}$ in
   $\F A$ is open, if and only if $L=K\cap A^*$
   for some clopen subset $K$ of $\F A$.
 \end{Thm}

 We next refer to some properties of pseudowords necessary along the paper.
 We begin with a cancellation property going back to~\cite[Exercise 10.2.10]{Almeida:1994a}. A recent proof can be found in~\cite[Solution to Exercise 4.20]{Almeida&ACosta&Kyriakoglou&Perrin:2020b}.
 
 \begin{Prop}\label{p:cut-common-suffix-prefix}
   Let $N$ be a positive integer. 
   If $u,v\in A^*$ have length $N$
   and $x,y\in\F A$ are such that $xu=yv$
   or $ux=vy$, then $x=y$ and $u=v$.
 \end{Prop}

 Therefore, for every $N\in\NN$, each infinite pseudoword $x$ of $\F A$  has a unique prefix and a unique suffix of length~$N$.
 In particular, taking $N=1$, every pseudoword which is not the empty word has a unique ``first letter'' and a unique ``last letter''.
 
For the sake of simplicity, we may write $\K_A$ instead of $\K_{\F A}$, whenever $\K$ is one of Green's relations $\R,\L,\J,\H$.
 
\begin{Cor}\label{c:finite-cut-and-you-are-still-in-the-j-class}
   Let $u$ be a regular element of $\F A$.
   If the factorization $u=vs$ is such that $s\in A^*$,
   then $u\mathrel{\R_A}v$.
 \end{Cor}

 \begin{proof}
   Let $z\in \F A$ be such that $u=uzu$.
   From $u=vs=uzvs$ and Proposition~\ref{p:cut-common-suffix-prefix}
   we get $v=uzv$, thus $u\mathrel{\R_A}v$.
 \end{proof}

 \begin{Cor}\label{c:finite-variation-conjugation-of-idempotents}
       Suppose that $yx$ is an idempotent in $\F A$, with $y\in A^*$.
       Then $xy$ is an idempotent of $\F A$  which is $\J_A$-equivalent to $yx$.
   \end{Cor}

   \begin{proof}
     From $yx=yxyx$ and
     Proposition~\ref{p:cut-common-suffix-prefix}
     we get $x=xyx$. Multiplying by $y$ on the left of both sides of
     $x=xyx$, we conclude that $xy$ is idempotent.
   \end{proof}

   \begin{Rmk}\label{rmk:yx-square-idempotent}
     If $x,y$ are arbitrary elements of a monoid $M$ such that $xy$ is idempotent, then $(yx)^2$ is idempotent, but in general $yx$ may not be idempotent.
   \end{Rmk}

   Extending what we do for subsets of $A^*$,
 we say that a subset $K$ of $\F A$ is \emph{factorial}
 if it contains the factors of elements of $K$;
 that it is \emph{prolongable}
 if for every $u\in K$
 there are $a,b\in A$ such that $au$ and $ub\in K$;
 and that $K$ is \emph{recurrent} if
 it is factorial, strictly contains the set $\{\varepsilon\}$, and $u,v\in K$
 implies the existence of some $w\in\F A$ such that $uwv\in K$ (note that $K$ is prolongable if $K$
 is recurrent).
For a proof of the next proposition, see~\cite[Proposition 5.6.1]{Almeida&ACosta&Kyriakoglou&Perrin:2020b}.
 
 \begin{Prop}\label{p:closure-of-factorials}
   Let $F$ be a subset of $A^*$.
   If $F$ is factorial/prolongable/recurrent subset of $A^*$,
   then $\overline{F}$ is a factorial/prolongable/recurrent subset of~$\F A$,
   respectively.
 \end{Prop}
 
 Let $K$ be a nonempty closed subset of $\F A$.
 Since the relation $\leq_{\J_A}$ is topologically closed in $\F A\times \F A$ and $\F A$ is a compact space, every $\leq_{\J_A}$-chain of elements of $K$ has a lower bound in $K$, and so, by Zorn's lemma,
 the set $K$ contains $\leq_{\J_A}$-minimal elements (that is, elements of $K$ that are minimal for the
 restriction of $\leq_{\J_A}$ to $K$).

 \begin{Rmk}\label{rmk:J-minimals}
   For every nonempty closed subset $K$ of $\F A$, every element of $K$ is a factor of a $\leq_{\J_A}$-minimal element of $K$: indeed, if $u\in K$, then the set $K_u=\{v\in K\mid v\leq_{\J_A}u\}$
 is itself nonempty and closed, and so it contains $\leq_{\J_A}$-minimal elements, which are clearly $\leq_{\J_A}$-minimal elements of~$K$.
 \end{Rmk}

 For a nonempty factorial language $F\subseteq A^*$, we denote by
 $J_A(F)$ the set of  $\leq_{\J_A}$-minimal elements of $\overline{F}\subseteq \F A$.
 Note that $J_A(F)$ is contained in $\overline{F}$ and is a union of $\J_A$-classes, since $\overline{F}$ is factorial (cf.~Proposition~\ref{p:closure-of-factorials}).
  If the alphabets $A,B$ satisfy $A\subseteq B$, then we see $\F A$ as a closed submonoid of $\F B$, and under that perspective we have $J_A(F)=J_B(F)$. Sometimes it will be convenient to take the smallest alphabet $A$ for which $F$ is a language of $A^*$. For a language $F\subseteq B^*$ containing nonempty words, we refer to the subset $A\subseteq B$ of letters that are factors of some element of $F$ as the \emph{alphabet of $F$}.

 If the language $F$ is recurrent over the alphabet $A$, then $J_A(F)$ is a regular $\J_A$-class \cite[Exercise 5.25]{Almeida&ACosta&Kyriakoglou&Perrin:2020b}.
 If $F$ is uniformly recurrent, then $J_A(F)$ is actually $\leq_{\J}$-maximal among regular $\J$-classes,
 as seen next. An infinite pseudoword $u$ in $\F A$ is
 a \emph{$\J$-maximal infinite pseudoword} if $u<_{\J_A} v$
 implies $v\in A^*$. 

 \begin{Thm}\label{t:jf-minimal-case}
   If $F\subseteq A^*$ is uniformly recurrent, then $\overline{F}=F\cup J_A(F)$.
   Moreover, the mapping $F\mapsto J_A(F)$
     is a bijection from the set of
     uniformly recurrent languages of~$A^*$
     onto the set of $\J_A$-classes of $\Cl J_A$-maximal infinite pseudowords of $\F A$.
   \end{Thm}

   Theorem~\ref{t:jf-minimal-case} is from
   {\cite{Almeida:2004a}}. A proof can be found in~\cite[Propositions 5.6.12 and 5.6.13]{Almeida&ACosta&Kyriakoglou&Perrin:2020b}.

\section{Codes}
\label{sec:codes}

A \emph{code} over the alphabet $A$ is a nonempty subset $X$ of $A^+$
that freely generates a submonoid of~$A^*$. For example, we have the \emph{prefix codes} (respectively, \emph{suffix codes}), that is, nonempty subsets $X$ of $A^+$
with no two elements $u,v$ of $X$ such that $u$ is a prefix
(respectively, suffix) of $v$.
A \emph{bifix code} is a code both prefix and suffix.

For dealing with the topological closure in $\F A$ of a rational code $X\subseteq A^+$, the following proposition is of great help.

\begin{Prop}[{\cite[Proposition 2.21]{Almeida&ACosta&Kyriakoglou&Perrin:2020}}]\label{p:right-unitary}
   Let $X$ be a rational code contained in $A^+$.
   In what follows, $u,v,w$
   are arbitrary pseudowords in $\F A$.
   If $X$ is a code, then the implication
     \begin{equation}\label{eq:right-unitary-1}
       u,vw,uv,w\in \overline{X^*}\Longrightarrow v\in\overline{X^*}
     \end{equation}
     holds. Moreover, if $X$ is a prefix code, then
     \begin{equation}\label{eq:right-unitary-2}
       u,uv\in\overline{X^*}\Longrightarrow v\in\overline{X^*}
     \end{equation}
     also holds. 
\end{Prop}

\subsection{Complete codes}

Consider a language $F\subseteq A^*$. A subset $X$ of $A^+$ is \emph{right $F$-complete}
if every element of $F$ is a prefix of an element of~$X^*$.
A prefix code $X$ contained in $F$
is an \emph{$F$-maximal prefix code}
if whenever $Y$ is a prefix code with $X\subseteq Y\subseteq F$, one has $Y=X$.
Combining the statements of \cite[Propositions 3.3.1 and 3.3.2]{Berstel&Felice&Perrin&Reutenauer&Rindone:2012}) we get the following.

\begin{Prop}\label{p:F-maximal}
  Let $F\subseteq A^*$ be a factorial language and let $X$ be a prefix code contained in~$F$.
  Then $X$ is an $F$-maximal prefix code if and only if it is
  right $F$-complete.
\end{Prop}

\begin{Example}\label{eg:example-of-prefix}
  The \emph{Fibonnaci language}
  over $\{a,b\}$ is the uniformly recurrent language $F=F_\varphi$
  induced by the primitive substitution $\varphi$ (the \emph{Fibonnaci substitution}) given by
  $\varphi(a)=ab$ and $\varphi(b)=a$. The prefix code $X=\{a,ba\}$
  is right $F$-complete.
\end{Example}

Of course, one has dual definitions
of left $F$-complete set, $F$-maximal suffix code, and the corresponding counterpart
of Proposition~\ref{p:F-maximal}.

A code $X$ contained in $F$ is \emph{$F$-complete}
when it is both right $F$-complete and left $F$-complete.
A bifix code $X$ contained in the language $F\subseteq A^*$
is an \emph{$F$-maximal bifix code}
if whenever $Y$ is a bifix code with $X\subseteq Y\subseteq F$, one has $Y=X$.
It turns out that if $F$ is recurrent, then the rational $F$-maximal bifix codes are precisely the rational $F$-complete bifix codes, as seen next.

\begin{Thm}\label{t:bifix-versus-prefix}
  Let $F$ be a recurrent language, and let $X$ be a rational bifix code contained in $F$.
  Then the following conditions are equivalent:
  \begin{enumerate}
   \item $X$ is an $F$-maximal bifix code;
  \item $X$ is left $F$-complete;
  \item $X$ is right $F$-complete;
  \item $X$ is $F$-complete.
  \end{enumerate}
\end{Thm}

Theorem~\ref{t:bifix-versus-prefix} is from the article~\cite[Theorem 4.2.2]{Berstel&Felice&Perrin&Reutenauer&Rindone:2012}.
In that paper it was assumed more generally that the bifix code $X$ is \emph{$F$-thin}, a property that is satisfied when
$X$ is rational (cf.~\cite[Proposition 2.8]{Almeida&ACosta&Kyriakoglou&Perrin:2020}).

When it is clear which alphabet $A$ we are talking about, when $F=A^*$ we may drop the ``$F$-'' in the previous definitions,
writing simply ``maximal prefix code'', ``complete code'', etc..

A \emph{group code} is a code $Z$ with alphabet $A$ for which the syntactic monoid of $Z^*$ is a finite group.\footnote{The definition of group code in the book~\cite{Berstel&Perrin&Reutenauer:2010} is more relaxed. We are following
      the definition of group code employed in the seminal paper~\cite{Berstel&Felice&Perrin&Reutenauer&Rindone:2012};
      the same definition is used in other sources that we cite, such as \cite{Almeida&ACosta&Kyriakoglou&Perrin:2020,Almeida&ACosta&Kyriakoglou&Perrin:2020b}. }
    Every group code is a rational maximal bifix code, see for example~\cite[Proposition~6.1.5]{Berstel&Felice&Perrin&Reutenauer&Rindone:2012}.

\begin{Example}\label{eg:group-code}
     For every positive integer~$n$,
    the language $A^n\subseteq A^+$ is a group code. 
\end{Example}

There is a close connection between complete and $F$-complete codes.

\begin{Thm}[{cf.~\cite[Theorems 4.2.11 and 4.4.3]{Berstel&Felice&Perrin&Reutenauer&Rindone:2012}}]\label{t:thin-code-intersects-F}
  Consider a recurrent
  language $F\subseteq A^*$.
  If $Z$ is a rational complete bifix code of $A^*$, then
  $Z\cap F$ is an $F$-complete bifix code.
  Moreover, if $F$ is uniformly recurrent, then $Z\cap F$ is finite.
\end{Thm}

\begin{Rmk}\label{rmk:converse-of-filtrations}
  Assume that $F\subseteq A^*$ is recurrent. By~\cite[Theorem 6.6.1]{Berstel&Perrin&Reutenauer:2010}, every rational bifix code $X\subseteq A^+$ is contained in a
  rational complete bifix code $Z\subseteq A^+$.
  For such $Z$, the bifix code $Z\cap F$ is $F$-complete by Theorem~\ref{t:thin-code-intersects-F}, thus an $F$-maximal bifix code by Theorem~\ref{t:bifix-versus-prefix}. Therefore, if $X$ is a rational $F$-complete bifix code, then $X=Z\cap F$ for some rational complete bifix code $Z$.
This fact is a sort of converse to Theorem~\ref{t:thin-code-intersects-F}. 
\end{Rmk}

For $L\subseteq A^*$, let $\eta_L$ be the syntactic homomorphism
from $A^*$ onto the syntactic monoid $M(L)$ of $L$.
 If $L$ is rational, then, as $M(L)$ is then finite,
 we may consider the unique continuous homomorphism
$\hat\eta_L\colon\F A\to M(L)$ extending $\eta_L$. In this paper we take advantage of the following fact.

\begin{Prop}\label{p:idempotents-are-in-the-closure-of-Xast}
  Let $F\subseteq A^*$. Suppose that $X$ is a rational $F$-complete bifix code.
  If $w$ is an element of $\overline{F}\subseteq \F A$ such that $\hat\eta_{X^*}(w)$
  is idempotent, then $w\in \overline{X^*}$.
\end{Prop}

A proof of Proposition~\ref{p:idempotents-are-in-the-closure-of-Xast} can be found
in \cite[Solution to Exercise~8.16]{Almeida&ACosta&Kyriakoglou&Perrin:2020b}.\footnote{In the statement of \cite[Exercise~8.16]{Almeida&ACosta&Kyriakoglou&Perrin:2020b}
  it is only implicit that the code $X$ is rational.}

\subsection{Charged codes}

Consider a recurrent language $F\subseteq A^*$.
Since $J_A(F)$ is a regular $\J$-class of $\F A$, it contains maximal subgroups, which are all isomorphic profinite groups.
Let $Z\subseteq A^+$ be a rational complete bifix code.
We say that $Z$ is \emph{$F$-charged} if for every maximal subgroup $K\subseteq  J_A(F)$ the image $\hat\eta_{Z^*}(K)$ is a maximal subgroup
contained in the minimum ideal of $M(Z^*)$ (the minimum ideal is the $\J$-class containing all $\leq_\J$-minimal elements). We also say that the rational code $X\subseteq A^+$ is \emph{$F$-complete $F$-charged} if
it is $F$-complete and there is an $F$-charged bifix code $Z\subseteq A^+$ such that
$X=Z\cap F$.

As a motivation for our main results, we exhibit until the end of this section examples of $F$-charged group codes.

We recall some basic notions of profinite group theory, available in~\cite{Ribes&Zalesskii:2010}.
Let~\pv H be a \emph{formation} of finite groups, i.e.~a class of finite groups closed under taking quotients and finite subdirect products,
and containing some nontrivial group. 
A \emph{pro-\pv H} group is an inverse limit of groups from \pv H, with onto connecting morphisms. 
The class of $A$-generated pro-\pv H groups has a free object $\FV HA$, the \emph{free pro-$\pv H$ group generated by $A$}.
If $\pv H$ has nontrivial groups, then $A$ embeds as a generating subset of the profinite group $\FV HA$.
Denote $p_{\pv H}$ the continuous onto homomorphism $\F A\to\FV HA$
that fixes the elements of $A$.

Let us say that a recurrent language $F$ with alphabet~$A$ is \emph{\pv H-charging} if
for every maximal subgroup $K$ of $J_A(F)$ (equivalently, for some maximal subgroup $K$ of $J_A(F)$)
the equality $p_{\pv H}(K)=\FV HA$ holds. Let us also say that a code $Z$ is an \emph{$\pv H$-code} if its syntactic monoid $M(Z^*)$ belongs to $\pv H$.

\begin{Prop}\label{p:group-surjective-implies-F-charged}
  Let $F\subseteq A^*$ be an \pv H-charging recurrent language. Then every~$\pv H$-code $Z\subseteq A^+$ is $F$-charged.
\end{Prop}

\begin{proof}
  By the universal property of $\FV HA$,
  the hypothesis $M(Z^*)\in\pv H$ implies
  that $\hat\eta_{Z^*}\colon \F A\to M(Z^*)$
  factorizes as $\hat\eta_{Z^*}=\bar\eta_{Z^*}\circ p_{\pv H}$,
  for some continuous onto homomorphism $\bar\eta_{Z^*}\colon \FV HA\to M(Z^*)$.
  Therefore, if $K$ is a maximal subgroup of $\F A$ contained in $J_A(F)$,
  then $\hat\eta_{Z^*}(K)=\bar\eta_{Z^*}(\FV HA)=M(Z^*)$.
\end{proof}

We denote by $\pv G$ the formation of all finite groups. The pro-$\pv G$ groups are the profinite groups, and the $\pv G$-codes are the group codes.

\begin{Example}\label{eg:dendric-group-codes-F-charged}
  The \emph{dendric} languages, also known as \emph{tree} languages,
  were introduced in~\cite{Berthe&Felice&Dolce&Leroy&Perrin&Reutenauer&Rindone:2015c}.
  Their definition is as follows. The \emph{extension graph}
  of a word $w$ in a language $F\subseteq A^*$ is the bipartite graph
  $G_F(w)$ where vertices are partitioned
  into disjoint copies $1\otimes L_F(w)$ and $R_F(w)\otimes 1$ of the sets $L_F(w)=\{a\in A\mid aw\in F\}$ and $R_F(w)=\{a\in A\mid wa\in F\}$,
  and where an edge from $1\otimes a$ to $b\otimes 1$ is a pair $(a,b)\in A\times A$ such that $awb\in F$.
  The set $F$ is \emph{dendric} (respectively, \emph{connected}) if $G_F(w)$ is a tree (respectively, connected graph) for every $w\in F$.
  In particular, every dendric language is connected.
  It turns out that every uniformly recurrent connected language is \pv G-charging, a property shown in~\cite[Theorem 2.19]{Almeida&ACosta&Kyriakoglou&Perrin:2020},
  basically repeating the proof in~\cite{Almeida&ACosta:2016b}
  for the dendric case.
 \end{Example}

For exhibiting another meaningful class of charged codes, it is convenient to introduce some more material about substitutions.
Let $\varphi$ be a primitive substitution over the alphabet $A$.
Let us say that $\varphi$ is \emph{stable} if there is $k\in\NN$
such that, for all $a,b\in A$, all factors of length two of $\varphi^k(ab)$ belong to $F_\varphi$ (equivalently,
if $x$ is the last letter of $\varphi^k(a)$ and $y$ is the first letter of $\varphi^k(b)$, then $xy\in F_\varphi$).

\begin{Example}
  The Fibonnaci substitution (Example~\ref{eg:example-of-prefix}) is stable, since we have $\varphi^2(\{a,b\}^2)\subseteq F_\varphi$.
\end{Example}

We proceed to give in the next lemma a ``profinite'' characterization of stable primitive substitutions, which is the characterization we need.
We denote by $\hat\varphi$
the unique continuous endomorphism of $\F A$ extending the endomorphism $\varphi\colon A^*\to A^*$. The monoid of continuous endomorphisms of a finitely generated profinite monoid is itself a profinite monoid for the pointwise topology (see~\cite[Sections 3.12 and 3.15]{Almeida&ACosta&Kyriakoglou&Perrin:2020b} for information about this result first obtained by Hunter~\cite{Hunter:1983}).
In particular, the idempotent
continuous idempotent endomorphism $\hat\varphi^\omega\colon \F A\to\F A$ in the following lemma makes sense.

\begin{Lemma}\label{l:stable-substitutions}
  A primitive substitution $\varphi\colon A^*\to A^*$ is stable if and only if $\hat\varphi^\omega(A^+)$ is contained in $J(F_\varphi)$.
\end{Lemma}

\begin{proof}
  Suppose $\varphi$ is stable.
  Let $u\in A^+$. We show by induction on the length of $u$
  that  $\hat\varphi^\omega(u)\in J(F_\varphi)$.
  The base case $u\in A$ holds for every primitive substitution, not necessarily stable (\cite[Theorem 3.7]{Almeida:2004a}, cf.~\cite[Section 5.6]{Almeida&ACosta&Kyriakoglou&Perrin:2020b}).

  Proceeding with the inductive step, assume that $\hat\varphi^\omega(u)\in J(F_\varphi)$.
  Let $b\in A$.
  To show that $\hat\varphi^\omega(ub)\in J(F_\varphi)$,
  we use the following criterion: an infinite pseudoword $w\in\F A$ belongs to $J(F_\varphi)$
  if and only if every finite factor of $w$ belongs to $F_\varphi$ (\cite[Corollary 2.8]{Almeida:2004a}, cf.~\cite[Section 5.6.9]{Almeida&ACosta&Kyriakoglou&Perrin:2020b}).
  Take a finite factor $v$ of $\hat\varphi^\omega(ub)$.
  Then $v$ is a finite factor of $\hat\varphi^\omega(u)$,
  or a finite factor of $\hat\varphi^\omega(b)$, or,
  if $a$ is the last letter of $u$,
  the concatenation of a finite suffix of $\hat\varphi^\omega(a)$
  with a finite prefix of $\hat\varphi^\omega(b)$, as words are placed in products of pseudowords
  as in the case of products of finite words (cf.~\cite[Example 4.4.19]{Almeida&ACosta&Kyriakoglou&Perrin:2020b}).

  We claim that $v\in F_\varphi$. If $v\in \hat\varphi^\omega(u)$ or $v\in \hat\varphi^\omega(a)$, then, by the induction hypothesis,
  $v$ is a factor of an element of $J(F_\varphi)$.
  Since the set $\overline{F_\varphi}$ is factorial and $J(F_\varphi)\subseteq\overline{F_\varphi}$ (cf.~Proposition~\ref{p:closure-of-factorials}), we immediately obtain $v\in F_\varphi$.

  It remains to look at the case where $v$ is the concatenation of a finite suffix of $\hat\varphi^\omega(a)$
  with a finite prefix of $\hat\varphi^\omega(b)$.  
  Since $\varphi$ is stable, there is $k\in\NN$ such that
  if $x$ is the last letter of $\varphi^k(a)$ and $y$ is the first letter of $\varphi^k(b)$,
  then $xy\in F_\varphi$. Note that $\varphi^{\omega-k}(x)$ is an infinite suffix
  of $\varphi^{\omega-k}(\varphi^k(a))=\varphi^\omega(a)$, and likewise $\varphi^{\omega-k}(y)$ is a infinite prefix of $\varphi^\omega(b)$.
  Therefore, $v$ is a factor of $\varphi^{\omega-k}(xy)$.
  As $\varphi(F_\varphi)\subseteq F_\varphi$ and since
  $xy\in F_\varphi$, we know that $\varphi^{\omega-k}(xy)\in\overline{F_\varphi}$,
  and so $v\in F_\varphi$, as $\overline{F_\varphi}$ is factorial.

  Hence, in all cases, $v\in F_\varphi$.
  We have thus shown that all finite factors of $\varphi^\omega(ub)$ belong to $F_\varphi$, which we already mentioned is the same as having
  $\varphi^\omega(ub)\in J(F_\varphi)$. This concludes the inductive proof of the ``only if'' part of the statement.

  Conversely, suppose $\hat\varphi^\omega(A^+)\subseteq J(F_\varphi)$.
  Take $a,b\in A$. Let $x$ and $y$ respectively be the last letter of $\varphi^\omega(a)$ and the first letter of $\varphi^\omega(b)$.
  Since $\overline{F_\varphi}$ is factorial and $xy$ is a factor of $\varphi^\omega(ab)$, it follows from our hypothesis that $xy\in F_\varphi$.  
  Because $\varphi^{n!}(a)\to\hat\varphi^\omega(a)$ and $\F A x$ is clopen,
  $x$ is a suffix of $\varphi^{n!}(a)$ for all sufficiently large $n$.
  Similarly, $y$ is a prefix of $\varphi^{n!}(b)$ for all sufficiently large $n$.
  As $xy\in F_\varphi$, this shows that every factor of length two of $\varphi^{n!}(ab)$ belongs to $F_\varphi$ for all sufficiently large $n$ and all $a,b\in A$. This establishes that $\varphi$ is stable.  
\end{proof}

A substitution $\varphi$ over the alphabet $A$ is said to be \emph{proper} if there are $b,c\in A$ such that $\varphi(A)\subseteq bA^*\cap A^*c$.
If $\varphi$ is a proper primitive substitution, then $\hat\varphi^\omega(A^+)$
is contained in a maximal subgroup of $J(F_\varphi)$ by~\cite[Proposition 5.3]{Almeida&Volkov:2006} (and in fact it is a maximal subgroup if $\varphi$ is nonperiodic, by~\cite[Lemma 6.3]{Almeida&ACosta:2013}). Therefore, every proper primitive substitution is stable,
according to Lemma~\ref{l:stable-substitutions}.

 For a formation \pv H of finite groups,
 denote $\hat\varphi_{\pv H}$ the unique continuous endomorphism of $\FV HA$ such that
 $\hat\varphi_{\pv H}(a)=p_{\pv H}(\varphi(a))$ for every $a\in A$.
 We say that $\varphi$ is \emph{$\pv H$-invertible} if $\hat\varphi_{\pv H}$ is an automorphism, a condition that is equivalent
 to have $(\hat\varphi_{\pv H})^\omega$ equal to the identity on $\FV HA$ (cf.~\cite[Proposition 3.7.4]{Almeida&ACosta&Kyriakoglou&Perrin:2020b}).

 \begin{Prop}\label{p:invertible-proper-substitution}
   Let $\pv H$ be a formation of finite groups.
   Suppose that $\varphi$ is an \pv H-invertible stable primitive  substitution. Then $F_\varphi$ is \pv H-charging. 
\end{Prop}

\begin{proof}
  First notice that, for every $u\in \FS A$, the projection $p_{\pv H}(u^\omega)$ is idempotent and thus is the neutral element
  of the group $\FV HA$. Hence, we have $p_{\pv H}(\FS A)=\FV HA$.
  
  By Lemma~\ref{l:stable-substitutions},
  the set $T=\hat\varphi^\omega(\FS A)$ is a closed subsemigroup
  of $\F A$ that is contained in the regular $\J_A$-class $J(F_\varphi)$. Therefore,
  $T$ is a completely simple profinite semigroup (i.e.,~$\J_T$ is the universal relation on $T$).
  
  Since the equality $p_{\pv H}\circ \hat\varphi=\hat\varphi_{\pv H}\circ p_{\pv H}$ holds,
  so does $p_{\pv H}\circ \hat\varphi^k=\hat\varphi_{\pv H}^k\circ p_{\pv H}$
  for every $k\in\NN$, whence
  $p_{\pv H}\circ \hat\varphi^\omega=\hat\varphi_{\pv H}^\omega\circ p_{\pv H}$.
  It follows that
  \begin{equation*}
  p_{\pv H}(T)=p_{\pv H}(\hat\varphi^\omega(\FS A))=(\hat\varphi_{\pv H})^\omega(p_{\pv H}(\FS A ))=(\hat\varphi_{\pv H})^\omega(\FV HA) =
   \FV HA,  
 \end{equation*}
 where the last equality holds because $\varphi$ is \pv H-invertible.
 Since $p_{\pv H}(T)=\FV HA$ and $T$ is completely simple, there is a closed subgroup $N$ contained in $T$ such that $p_{\pv H}(N)=\FV HA$ (cf.~\cite[Lemma 4.6.10]{Rhodes&Steinberg:2009qt}).
 Therefore, if $K$ is the maximal subgroup of $J(F_\varphi)$ containing $N$, then $p_{\pv H}(K)=\FV HA$.
\end{proof}

The subgroup of $\FV GA$ generated by $A$ is the $A$-generated free group $FG(A)$.
It is well known that a substitution $\varphi\colon A^*\to A^*$ is $\pv G$-invertible if and only if the endomorphism
of $FG(A)$ extending $\varphi$ has an inverse, if and only if $FG(A)$ is generated by $\varphi(A)$~\cite[Proposition 4.6.8]{Almeida&ACosta&Kyriakoglou&Perrin:2020b}. We use this in the next example.

\begin{Example}
  Consider the alphabet $A=\{0,1,2\}$, and let $\varphi$ be the proper (and thus stable) primitive substitution~$\varphi$
  over $A$ given by
  \begin{equation*}
    \varphi(0)=012,\quad \varphi(1)=0122,\quad \varphi(2)=0121012.
  \end{equation*}
  Then $\varphi$ is $\pv G$-invertible. Therefore, by Propositions~\ref{p:group-surjective-implies-F-charged}
  and~\ref{p:invertible-proper-substitution}, every group code over the alphabet $A$ is $F_\varphi$-charged.
  
  The language $F_\varphi$ is not connected: for example, the extension graph
  of the letter~$1$ is disconnected. Indeed,
  $F_\varphi\cap A1A=\{012,210\}$,
  since for every $a\in A$ the letter $1$
  is a factor of $\varphi(a)$ which is neither
  a prefix nor a suffix of $\varphi(a)$,
  and so the elements of $F_\varphi\cap A1A$
  must be factors of some of the words $\varphi(0)$, $\varphi(1)$, $\varphi(2)$.
\end{Example}

The matrix $M_\varphi$ associated to a substitution $\varphi\colon A^*\to A^*$
is the matrix $A\times A$ where each entry $(a,b)$ is the number of occurrences of $a$ in $\varphi(b)$.
For a set $\pi$ of prime numbers, the formation of finite nilpotent $\pi$-groups is denoted $\pv G_{\pv{nil},\pi}$.
For the next example we take advantage of the fact that $\varphi$ is
$\pv G_{\pv{nil},\pi}$-invertible if and only if $\det M_\varphi\not\equiv 0\pmod p$ for every $p\in \pi$ (see the proof of~\cite[Corollary 5.3]{Almeida:2001b}).

\begin{Example}
  Let $\varphi$ be the proper primitive substitution
  over $A=\{0,1\}$ such that $\varphi(0)=01$ and $\varphi(1)=0001$.
  As $\det M_\varphi=-2$, if $\pi$ is the set of odd primes, then $\varphi$ is $\pv G_{\pv{nil},\pi}$-invertible. Hence, by Propositions~\ref{p:group-surjective-implies-F-charged}
  and~\ref{p:invertible-proper-substitution}, every group code $Z\subseteq A^+$ such that the finite group $M(Z^*)$ is nilpotent of odd order is $F_\varphi$-charged.
  
  In contrast, the group code $Z=A^2$ is not $F_\varphi$-charged: indeed, $M(Z^*)=\ZZ/{2\ZZ}$, and there is a maximal subgroup $K$ contained
  in the image of $\hat\varphi$ by~\cite[Lemma 6.3]{Almeida&ACosta:2013},
  thus $\hat\eta_{Z^*}(\hat\varphi(\F A))=0=\hat\eta_{Z^*}(K)$ as $\varphi(0),\varphi(1)$ have even length.
\end{Example}

Several other examples of $F$-charged complete bifix codes $Z$ are given in~\cite{Almeida&ACosta&Kyriakoglou&Perrin:2020}, in which $Z$ may not be a group code, or $F$ is recurrent but not uniformly recurrent.

\section{Decoding of languages}
\label{sec:decoding-languages}

From hereon, $X$ is a finite code contained in $A^+$.
A \emph{coding morphism} for $X$ is an injective homomorphism~$\beta_X\colon B^*\to A^*$
such that $\beta_X(B)=X$, for some alphabet~$B$.

\begin{Example}\label{eg:recurrence-not-preserved-by-decoding}
   Take $A=\{a,b\}$, $X=A^2$,
   and consider the coding morphism
   $\beta_X\colon B^*\to A^*$ given by $\beta_X(z)=ab$
   and $\beta_X(t)=ba$, where $B=\{z,t\}$.
   Let $F$ be the set of factors of $(ab)^*$. Then $\beta_X^{-1}(F)=z^*\cup t^*$.
 \end{Example}

   The study of coding morphisms reduces to the case where $B=X$ and $\beta_X$ is the inclusion $X^*\to A^*$,
   since such assumption reflects a mere relabeling of letters.
   If $\beta_X$ is indeed the inclusion $X^*\to A^*$,
   then $\beta_X^{-1}(F)$ is the intersection $F\cap X^*$.
   The set $F\cap X^*$ is the \emph{decoding} of $F$ by~$X$.
   If $X$ is an $F$-complete bifix code, then $F\cap X^*$ is said to be a \emph{complete bifix decoding}.

     Take a factorial language $F\subseteq A^*$ and a code $Z\subseteq A^+$.
     Set $X=F\cap Z$. Then the equality
     \begin{equation*}
      F\cap Z^*=F\cap X^* 
     \end{equation*}
     holds.
     Therefore, by Theorem~\ref{t:thin-code-intersects-F} and Remark~\ref{rmk:converse-of-filtrations},
     if $F$ is uniformly recurrent, then a set is a complete bifix decoding of $F$ by a finite $F$-complete bifix code
     if and only if it is of the form $F\cap Z^*$
     for some rational complete bifix code~$Z$.
     In symbolic dynamics, the following example of a complete bifix decoding process is of great importance (cf.~\cite[Section 1.4]{Lind&Marcus:1996}).
 
 \begin{Example}\label{eg:higher-power-shift}
   Let $\Cl S$ be a subshift of $A^\ZZ$. Consider a positive integer $n$.
   Then $Z=A^n$ is a complete bifix code (actually, a group code).
Let $F=\Cl B(\Cl S)$. The \emph{$n$-th higher power} of $\Cl S$ is the subshift $\Cl S^n$ of $X^\ZZ$ defined by the equality $\Cl B(\Cl S^n)=F\cap Z^*$. 
 \end{Example}
 
Margolis, Sapir and Weil obtained the following
key result~\cite{Margolis&Sapir&Weil:1995}.

\begin{Thm}\label{t:injective-extension}
  Let $X$ be a finite code contained in $A^+$. The unique extension of a coding morphism $\beta_X\colon B^*\to A^*$ to a continuous
homomorphism $\bar\beta_X\colon \F B\to\F A$ is an injective mapping.
\end{Thm}
    
  Under the assumption that $\beta_X$ is the inclusion, Theorem~\ref{t:injective-extension} says that, for every finite code $X\subseteq A^+$,
  the free profinite monoid over $X$
  is identified with the closed submonoid of $\F A$ generated by $X^*$, that is,
  the equality
  \begin{equation*}
    \F X=\overline{X^*}
  \end{equation*}
  holds whenever $X$ is a finite code. The reader should bear in mind this equality along the paper. Notice also that $\F X$ is a clopen subset of $\F A$, by Theorem~\ref{t:rational-open}.
  Therefore, the equality
  \begin{equation*}
    \overline{F\cap X^*}=\overline{F}\cap\F X
  \end{equation*}
  holds, for every subset $F$ of $A^*$. (The topological closure in $\F X$ of a subset $L$ of $\F X$ coincides with the topological closure of $L$ in $\F A$, and for that reason the notation $\overline{L}$ is not ambiguous.)
  
  Clearly, for every finite code $X\subseteq A^+$,
  if $F$ is a factorial language over the alphabet $A$,
  then $F\cap X^*$ is a factorial langage over the alphabet $X$;
  and if $F$ is prolongable over $A$ and $X$ is $F$-complete,
  then $F\cap X^*$ is prolongable over $X$.
  Therefore, if $F$ is a subshift language over $A$ and $X$ is $F$-complete, then $F\cap X^*$ is a subshift language over $X$.
 
  \begin{Example}\label{eg:fails-for-prefix}
    Here is an example where $F\cap X^*$ is not be prolongable over~$X$:
    for the Fibonacci language $F$ and the right $F$-complete prefix code $X=\{a,ba\}$,
    we have $a^2\in F\cap X^*$, but there is
  no $x\in X$ such that $xa^2\in F$.
\end{Example}
 
A complete bifix decoding of a uniformly recurrent language may not be recurrent, as seen in Example~\ref{eg:recurrence-not-preserved-by-decoding}.

If $X$ is a finite $F$-complete $F$-charged code, then we say that $F\cap X^*$ is a \emph{charged complete bifix decoding}.
In contrast with Example~\ref{eg:recurrence-not-preserved-by-decoding},
we have the following theorem, one of the main results of this paper.

 \begin{Thm}\label{t:charged-decoding-uniformly-recurrent}
    Every charged complete bifix decoding
    of a uniformly recurrent language is uniformly recurrent.
  \end{Thm}

  We defer to Section~\ref{sec:proof-theor-reft} the proof of Theorem~\ref{t:charged-decoding-uniformly-recurrent}.
  We also defer (to Section~\ref{sec:proof-theor-reft:recurrent}) the proof of the following analog of Theorem~\ref{t:charged-decoding-uniformly-recurrent}.
  
  \begin{Thm}\label{t:charged-decoding-recurrent}
    Every charged complete bifix decoding
    of a recurrent language is recurrent.
  \end{Thm}

  We let $\ell_n^A$ denote the unique continuous homomorphism
  $\F A\to \ZZ/{n\ZZ}$ such that
  $\ell_n^A(a)$ is the class of $1$ modulo $n$ for every letter $a\in A$.
  For an irreducible subshift $\Cl S\subseteq A^\ZZ$, we let
  $J(\Cl S)$ denote the $\J_A$-class $\J_A(\Block(\Cl S))$.

  \begin{Cor}\label{c:the-image-is-the-cyclic-group-modulo-n}
  Let $\Cl S$ be an irreducible subshift of $A^{\ZZ}$.
  Let $n$ be a positive integer, and let $X=\Block(S)\cap A^n$.
  Suppose that  for some maximal subgroup $H$ of $J(\Cl S)$ we have $\ell_n^A(H)=\ZZ/{n\ZZ}$. Then the subshift $\Cl S^n\subseteq X^\ZZ$  is irreducible.
  If moreover $\Cl S$ is minimal, then $\Cl S^n$
  is a minimal subshift of $X^\ZZ$.
\end{Cor}

\begin{proof}
  Consider the group code $Z=A^n$.
  The syntactic homomorphism of $Z^*$ is precisely
  the restriction $\ell_n^A|_{A^*}\colon A^*\to\ZZ/{n\ZZ}$. Since $\ell_n^A(H)=\ZZ/{n\ZZ}$ for a maximal subgroup $H\subseteq J(\Cl S)$, the complete bifix code $Z$ is $\Block(\Cl S)$-charged.
  Hence $\Block(\Cl S)\cap X^*$ is a recurrent language over the alphabet $X$, by Theorem~\ref{t:charged-decoding-recurrent}, that is, the subshift
  $\Cl S^n\subseteq X^\ZZ$ is irreducible.
  If moreover $\Cl S$ is minimal,
  then so is $\Cl S^n$ by Theorem~\ref{t:charged-decoding-uniformly-recurrent}.
\end{proof}
    
  \section{Finite factors of idempotents of $J_A(F)\cap \F X$}

For each pseudoword $v\in\F A$, let $\Fin_A(v)$
 be the subset of $A^*$ consisting of the finite factors in $\F A$ of~$v$.
 Clearly, if $v$ is a factor of $w$, then $\Fin_A(v)\subseteq\Fin_A(w)$,
 thus $\Fin_A(u)=\Fin_A(v)$ whenever~$u\mathrel{\J_A}v$.
 
 \begin{Lemma}\label{l:finite-factors-factorial-prolongable}
    If $e$ is an infinite idempotent of $\F A$, then $\Fin_A(e)$
    is a subshift language of $A^*$.
  \end{Lemma}

  \begin{proof}
    It is trivial that $\Fin_A(e)$ is a nonempty factorial language.
    Let $u\in\Fin_A(e)$. Then we have $e=xuy$ for some $x,y\in\F A$.
    Since $e$ is idempotent, we may suppose that $x=ex$ and $y=ye$.
    Therefore, $x$ has a last letter $a$, and $y$ has a first letter $b$,
    and $aub\in\Fin_A(e)$, showing that $\Fin_A(e)$ is prolongable. 
  \end{proof}
 
  Denote $\Pref_A(v)$ the set of prefixes of the pseudoword $v\in\F A$
  belonging to $A^*$.
  
\begin{Lemma}[{cf.~\cite[Exercise 5.28]{Almeida&ACosta&Kyriakoglou&Perrin:2020b}}]\label{l:prefix-of-some-idempotent-in-JF}
  Let $F$ be a recurrent language over the alphabet $A$. If $v\in F$, then
  there is an idempotent $e\in J_A(F)$ such that $v\in\Pref_A(e)$.
 \end{Lemma}

 For a subshift language $F\subseteq A^*$, we denote by $E_A(F)$ the set of idempotents in $J_A(F)$,
  and by $E_X(F)$ the set of idempotents in $J_A(F)\cap\F X$.

  \begin{Prop}\label{p:Fx-in-the-bifix-case}
    Let $F\subseteq A^*$ be a recurrent language, and let $X$ be a finite code.
    Then we have the inclusions
        \begin{equation}\label{eq:Fx-in-the-bifix-case-1}
      F\cap X^*\supseteq\bigcup_{e\in E_X(F)}\Fin_X(e)\supseteq\bigcup_{e\in E_X(F)}\Pref_X(e).
    \end{equation}
    Moreover, if $X$ is a finite $F$-complete bifix code, then
    $F\cap X^*$ is a subshift language, and the equalities
    \begin{equation*}
      E_X(F)=E_A(F)
    \end{equation*}
     and
    \begin{equation*}
      F\cap X^*=\bigcup_{e\in E_A(F)}\Fin_X(e)=\bigcup_{e\in E_A(F)}\Pref_X(e)
    \end{equation*}
    hold.
\end{Prop}

\begin{proof}
  For $e\in E_X(F)$,
  it is trivial that $\Pref_X(e)\subseteq \Fin_X(e)$. Since $\overline{F}$ is
  a factorial subset of $\F A$ (cf.~Proposition~\ref{p:closure-of-factorials}), we have $ \Fin_X(e)\subseteq F\cap X^*$,
  yielding~\eqref{eq:Fx-in-the-bifix-case-1}.

  Suppose that $X$ is a finite $F$-complete bifix code.
  By Proposition~\ref{p:idempotents-are-in-the-closure-of-Xast},
  every idempotent in $\overline{F}$ belongs to $\overline{X^*}$, and so $E_A(F)=E_X(F)$ as $J_A(F)\subseteq\overline{F}$.
  By~\eqref{eq:Fx-in-the-bifix-case-1},
  it remains to establish $F\cap X^*\subseteq \bigcup_{e\in E_A(F)}\Pref_X(e)$.
  Let $u\in F\cap X^*$.
  Then there is an idempotent $e\in E_A(F)$ with $e=uw$ for some $w\in \F A$, by Lemma~\ref{l:prefix-of-some-idempotent-in-JF}.
  Since $E_A(F)=E_X(F)\subseteq\overline{X^*}$, it then follows
  from implication~\eqref{eq:right-unitary-2}
  in Proposition~\ref{p:right-unitary}
  that $w\in \overline{X^*}$. Hence the factorization
  $e=uw$ yields $u\in\Pref_X(e)$.
\end{proof}

Proposition~\ref{p:Fx-in-the-bifix-case} motivates the consideration of the
subset $F_X$ of $F\cap X^*$ given by
\begin{equation*}
  F_X=\bigcup_{e\in E_X(F)}\Fin_X(e),
\end{equation*}
where $X\subseteq A^+$ is a finite code and $F\subseteq A^*$ is a recurrent language. 
In what follows, $||X||$ is the length of the element of $X$ with greatest length.

  \begin{Prop}\label{p:decoding-is-a-union}
    Suppose that $F$ is a recurrent language over the alphabet $A$ and that~$X$ is
    a finite code contained in $F$. Then, the union
    \begin{equation*}
   F_X=\bigcup_{e\in E_X(F)}\Fin_X(e)
 \end{equation*}
 is a finitary union: the set $\{\Fin_X(e)\mid e\in E_X(F)\} $
 has at most $||X||$ elements.
\end{Prop}

For the proof of this proposition, we use the following fact.

\begin{Lemma}\label{l:Xstar-factorization}
  Let $X$ be a finite subset
  of $A^*$. Let $\rho=uv$
  be a factorization of a pseudoword $\rho$
  in $\overline{X^*}$.
  Then we have $u=\mu x$, $v=y\nu $ for some pseudowords
  $\mu,\nu\in\overline{X^*}$,
  and some $x,y\in A^*$ such that $xy\in X$.
\end{Lemma}

\begin{proof}
  Let $(u_n)_n$ and $(v_n)_n$ be sequences
  of elements of $A^*$ such that $u_n\to u$ and $v_n\to v$.
  Since $X^*$ is rational, the set $\overline{X^*}$ is an open neighborhood of $\rho=uv$.
  Hence, by taking subsequences, we may as well suppose that $u_nv_n\in X^*$
  for all $n$. We then have factorizations $u_n=\mu_n x_n$
  and  $v_n=y_n\nu_n$
  with $\mu_n,\nu_n\in X^*$ and $x_ny_n\in X$, for every $n$.
  Let $(\mu,x,y,\nu)$
  be an accumulation point
  of the sequence $(\mu_n,x_n,y_n,\nu_n)$.
  Then $u=\mu x$, $v=y\nu$ and, since $X$ is finite, $x,y$
  are words such that $xy\in X$.
\end{proof}

\begin{proof}[Proof of Proposition~\ref{p:decoding-is-a-union}]
  Given any infinite idempotent $f$ of $\F A$,
  and any nonnegative integer $n$,
  there is a unique prefix $y$ of length~$n$
  of $f$, and for such $y$, there is a unique
  suffix $f'$ of $f$ with $f=yf'$, according to Proposition~\ref{p:cut-common-suffix-prefix}.
  In what follows, we denote $\sigma^n(f)$ the pseudoword $f'y$,
  which is idempotent
  by Corollary~\ref{c:finite-variation-conjugation-of-idempotents}.
    
  Fix $f\in E_X(F)$. Let $e$ be an arbitrary element of $E_X(F)$.
    Take $p,q\in\F A$ such that $e=pfq$.
    Since $f$ is idempotent, we may take~$p$ and~$q$ such that $p=pf$ and $q=fq$,
    so that $e=pq$.
    Because $e\in\overline{X^*}$,
    applying Lemma~\ref{l:Xstar-factorization}
    we obtain $p=p'x$
    and $q=yq'$
    for  some $p',q'\in\overline{X^*}$ and
    words $x,y$
    such that $xy\in X$.
    We claim that we may suppose that $x\neq\varepsilon$: indeed, since $p'$ is an infinite pseudoword,
    we have in fact $p'=p'' x'$ for some $p''\in\overline{X^*}$
    and $x'\in X$,
    so that if $x=\varepsilon$ (and thus $y\in X$)
    we could then replace $x$ by $x'$, $p'$ by $p''$, $y$ by $\varepsilon$,
    and $q'$ by $q$, showing the claim.
    Throughout the proof, we suppose that $x\neq\varepsilon$.

    From $fq=q=yq'$  we obtain a factorization
    $f=yf'$ for some (unique) infinite pseudoword $f'$, by
    Proposition~\ref{p:cut-common-suffix-prefix}.
    From $yf'q=fq=yq'$
    we get~$f'q=q'$ by cancellation of the finite prefix $y$, again by Proposition~\ref{p:cut-common-suffix-prefix}.
    Let $n=|y|$. Note that
    \begin{equation*}
      \sigma^n(f)=f'y.
    \end{equation*}
    Since $x\neq \varepsilon$, we have $n<||X||$.
        Note also that
    \begin{equation*}
      p'xy\cdot f'y=pfy=py=p'xy\in\overline{X^*}\qquad\text{and}
      \qquad f'y\cdot q'=f'q=q'\in\overline{X^*}.
    \end{equation*}
    By the implication~\eqref{eq:right-unitary-1}
    in  Proposition~\ref{p:right-unitary}, with $u=p'xy$, $v=f'y$ and $w=q'$,
    we conclude that $f'y\in\overline{X^*}$,
    that is $f'y\in E_X(F)$.
    Hence, in the factorization 
    \begin{equation}\label{eq:decoding-is-a-union-1}
     e=py\cdot f'y\cdot q'
   \end{equation}
   all factors $py=p'xy$, $f'y$ and $q'$
   belong to $\overline{X^*}$,
   whence
   \begin{equation}
     \label{eq:e-below-f-prime-y}
     e\leq_{\J_X}f'y.
   \end{equation}

    By~\eqref{eq:decoding-is-a-union-1} and Corollary~\ref{c:finite-variation-conjugation-of-idempotents},
    the pseudoword $f''=(f'y\cdot q'\cdot py)^2$
    is idempotent (cf.~Remark~\ref{rmk:yx-square-idempotent}).
    Moreover, we have $e\mathrel{\J_A}f''\mathrel{\J_A}f''\cdot f'y$
    and $f''\cdot f'y\in f'y\cdot \F A\cdot f'y$.
    Since the monoid $\F A$ is stable, we conclude that the pseudoword $f''\cdot f'y$ belongs
    to the maximal subgroup of $J_A(F)$ containing the idempotent~$f'y$.
    It follows that $f'y=(f''\cdot f'y)^\omega=(f'y\cdot q'\cdot py\cdot f'y)^\omega$.
    Then, the next chain of equalities holds:
    \begin{align*}
      f'y
      &=f'y\cdot q'\cdot(py\cdot f'y\cdot q')\cdot py\cdot (f'y\cdot q'\cdot py\cdot f'y)^{\omega-2}\\
      &=f'y\cdot q'\cdot e\cdot py\cdot (f'y\cdot q'\cdot py\cdot f'y)^{\omega-2}.
    \end{align*}
    In the former expression, all factors between consecutive dots
    belong to $\overline{X^*}$,
    and so, combining with~\eqref{eq:e-below-f-prime-y}, we conclude
    that
    \begin{equation*}
     \sigma^n(f)=f'y\mathrel{\J_X}e.
    \end{equation*}
    This shows that $\Fin_X(e)=\Fin_X(\sigma^n(f))$.
    Therefore, the equality
    \begin{equation*}
      F_X=\bigcup_{\underset{\sigma^n(f)\in E_X(F)}{n\leq||X||}}\Fin_X(\sigma^n(f)),
    \end{equation*}
    holds, and the proposition is established.
\end{proof}

The set $F_X$ is always a factorial and prolongable language over the alphabet $X$ (cf.~Lemma~\ref{l:finite-factors-factorial-prolongable}),
but $F_X$ may be empty since $E_X(F)$ may be empty.

\begin{Example}\label{eg:EXF-may-be-empty}
  Let $F$ be the Fibonacci language and let $X$ be the bifix code $\{aa,bb\}$. Then $F\cap X^*$ is finite, as seen in~\cite[Example 3.12]{Berthe&Felice&Dolce&Leroy&Perrin&Reutenauer&Rindone:2015}. Because
  $\Fin_X(e)$ is infinite whenever $e\in E_X(F)$ (cf.~Lemma~\ref{l:finite-factors-factorial-prolongable}),
  the inclusion of $F_X$ on the finite set $F\cap X^*$
  implies that $E_X(F)=\emptyset$ and $F_X=\emptyset$.
\end{Example}

In contrast with Example~\ref{eg:EXF-may-be-empty}, we have the following.

\begin{Prop}\label{p:tilde-FX-not-empty}
  Let $F\subseteq A^*$ be a recurrent language.
  If~$X$ is a right $F$-complete
  prefix code, then $F_X$ is nonempty.
\end{Prop}

\begin{proof}
  Let $(u_n)_n$ be a sequence of elements of $F$
  converging in $\F A$ to $u\in J_A(F)$.
  As $X$ is right $F$-complete, for each $n$  we have a factorization
  $u_n=x_ns_n$
  with $x_n\in X^*$ and $s_n\in A^*$
  such that $s_n$ is a proper prefix
  of an element of $X$.
  Take an accumulation point $(x,s)$ of $(x_n,s_n)$.
  Then $u=xs$ and, since $X$ is finite, one has $s\in A^*$. It follows from
  Corollary~\ref{c:finite-cut-and-you-are-still-in-the-j-class} that
  $x\in J_A(F)\cap\overline{X^*}$.
  Hence $x$ is regular and $x=xe$ for some idempotent $e\in J_A(F)$.
  Since $X$ is a prefix code,
  applying Proposition~\ref{p:right-unitary} we get $e\in E_X(F)$,
  and so $E_X(F)$ and $F_X$ are nonempty.
\end{proof}

  Proposition~\ref{p:Fx-in-the-bifix-case} guarantees that $F_X=F\cap X^*$
  when $X$ is a finite $F$-complete bifix code.
  The inclusion $F_X\subseteq F\cap X^*$ may be strict when $X$ is just a right $F$-complete prefix code,
  as $F\cap X^*$ may not be prolongable over $X$ (cf.~Example~\ref{eg:fails-for-prefix}).

\section{Relationship between $J_A(F)$ and $J_X(F\cap X^*)$}
In this section we seek to clarify the relationship between the $\J_A$-class $J_A(F)$ and the union of $\J_X$-classes $J_X(F\cap X^*)$ when $F$ is recurrent and $X$ is a finite $F$-complet bifix code.

The following lemma helps dealing with factorizations of pseudowords.

\begin{Lemma}\label{l:refine-factorization}
   Let $u,v\in \F A$.
   If the sequence $(w_n)_n$ converges to $uv$ in $\F A$,
   then there are factorizations $w_n=u_nv_n$
   such that $u_n\to u$ and $v_n\to v$.
 \end{Lemma}

 A proof of this lemma appears in~\cite[Exercise 4.20]{Almeida&ACosta&Kyriakoglou&Perrin:2020b}; the result
 was discovered independently in \cite{Almeida&ACosta:2007a} and \cite{Henckell&Rhodes&Steinberg:2010b} (cf.~ also~\cite[Section 3]{Almeida&ACosta&Costa&Zeitoun:2019}).
 We apply the lemma in the proof of the next proposition, the bulk of this section.

\begin{Prop}\label{p:J-minimal-elements-are-almost-in-the-same-J-class}
  Let $F$ be a recurrent language over the alphabet $A$, and suppose that $X$ is a finite $F$-complete bifix code.
  Then we have the following properties:
  \begin{enumerate}[label=(\roman*)]
  \item the equality $J_X(F\cap X^*)=J_A(F)\cap \F X$ holds;\label{item:J-minimal-elements-are-almost-in-the-same-J-class-1}
  \item $J_X(F\cap X^*)$ is a union of finitely many regular $\J_X$-classes of $\F X$;\label{item:J-minimal-elements-are-almost-in-the-same-J-class-2}
  \item for each $\K\in\{\R,\L,\H\}$, the $\K_X$-classes contained in $J_X(F\cap X^*)$
    are the sets of the form $V\cap\F X$
    that are nonempty and such that $V$ is a $\K_A$-class contained in $J_A(F)$;\label{item:J-minimal-elements-are-almost-in-the-same-J-class-3}
  \item the maximal subgroups of $\F X$ contained in $J_X(F\cap X^*)$
    are the intersections $H\cap\F X$
    in which $H$ is a maximal subgroup of $\F A$ contained in $J_A(F)$;\label{item:J-minimal-elements-are-almost-in-the-same-J-class-3b}
  \item if $yX^*x\cap F\subseteq X^*$ whenever $xy\in X^*$, then the maximal subgroups of $\F X$ contained in $J_X(F\cap X^*)$ are isomorphic profinite groups.\label{item:J-minimal-elements-are-almost-in-the-same-J-class-4}
  \end{enumerate}
\end{Prop}

\begin{Rmk}\label{rmk:examples-where-J-minimal-elements-are-almost-in-the-same-J-class-applies}
  Proposition~\ref{p:J-minimal-elements-are-almost-in-the-same-J-class} holds if $X=F\cap A^n$;
  more generally, if $Z$ is a group code with $Z^*=\eta_{Z^*}^{-1}(N)$
  for a normal subgroup $N$ of $M(Z^*)$, then Proposition~\ref{p:J-minimal-elements-are-almost-in-the-same-J-class}
  applies to $X=F\cap Z$. Indeed,  $X=Z\cap F$ is then finite $F$-complete (cf.~Theorem~\ref{t:thin-code-intersects-F}), and
  $xy\in Z^*\Leftrightarrow\eta_{Z^*}(y)N=\eta_{Z^*}(x)^{-1}N$,
  yielding
  \begin{equation*}
  xy\in Z^*\Rightarrow \eta_{Z^*}(yZ^*x)N=\eta_{Z^*}(y)N\eta_{Z^*}(x)N=N\Rightarrow yZ^*x\subseteq \eta_{Z^*}^{-1}(N)=Z^*,  
\end{equation*}
and thus $yX^*x\cap F\subseteq X^*$, ensuring that Proposition~\ref{p:J-minimal-elements-are-almost-in-the-same-J-class}\ref{item:J-minimal-elements-are-almost-in-the-same-J-class-4} holds for such $X$.  
\end{Rmk}

We proceed with the proof of Proposition~\ref{p:J-minimal-elements-are-almost-in-the-same-J-class}. For showing its item~\ref{item:J-minimal-elements-are-almost-in-the-same-J-class-4} we adapt the proof that maximal subgroups
in a regular $\J$-class of a compact monoid are all isomorphic compact groups, see  for instance \cite[Proposition 3.6.11]{Almeida&ACosta&Kyriakoglou&Perrin:2020b}.

\begin{proof}[Proof of Proposition~\ref{p:J-minimal-elements-are-almost-in-the-same-J-class}]
  Let $u\in J_X(F\cap X^*)$.
  There is a finite set $I$ of idempotents
  in $E_X(F)=J_A(F)\cap\F X$ such that $F\cap X^*=\bigcup_{e\in I}\Fin_X(e)$, by Propositions~\ref{p:Fx-in-the-bifix-case} and~\ref{p:decoding-is-a-union}.
  Hence,  as $J_X(F\cap X^*)\subseteq\overline{F\cap X^*}$,
  there is $e\in I$ with $u\in \overline{\Fin_X(e)}$.
 Since $\leq_{\J_X}$ is a closed relation of $\F X$,
  we have $e\leq_{\J_X}u$. Note that $e\in\overline{F}\cap \F X=\overline{F\cap X^*}$, as $J_A(F)\subseteq \overline{F}$.  It follows from
  the $\J_X$-minimality of $u$ as an element of $\overline{F\cap X^*}$ that $e\mathrel{\J_X}u$.
  This shows~\ref{item:J-minimal-elements-are-almost-in-the-same-J-class-2},
  as well the inclusion $J_X(F\cap X^*)\subseteq J_A(F)\cap \F X$ in~\ref{item:J-minimal-elements-are-almost-in-the-same-J-class-1}.

  Let $u\in J_A(F)\cap \F X$. Since $F$ is recurrent, the set
  $J_A(F)$ is a regular $\J_A$-class, and so we may take some
  idempotent $e\in J_A(F)$ that is $\R_A$-equivalent to
  $u$. Let $s\in \F A$ be such that $e=us$.
  As $X$ is an $F$-complete bifix code,
  we know that $e\in\F X$ by Proposition~\ref{p:idempotents-are-in-the-closure-of-Xast}, which together with the equality $e=us$ implies
  $s\in \F X$ by Proposition~\ref{p:right-unitary}. Since we also have $u=eu$,
  we deduce that
  $u\mathrel{\R_X}e$.
  We claim that $e\in J_X(F\cap X^*)$.
  Since $e\in J_A(F)\cap \F X\subseteq \overline{F\cap X^*}$,
  and we already showed that $J_X(F\cap X^*)$
  is a union of regular $\J_X$-classes, we know that there is some idempotent
  $f\in J_X(F\cap X^*)$ and some $x,y\in \F X$
  such that $f=xey$ (cf.~Remark~\ref{rmk:J-minimals}).
  Then the idempotent $e'=(eyx)^2$
  is $\J_X$-equivalent to $f$
  and satisfies $e'\leq_{\R_X}e$,
  whence $e'\leq_{\R_A}e$.
  Since $e, e'\in J_A(F)$ as
  the inclusion $J_X(F\cap X^*)\subseteq J_A(F)$ was already established,
  by stability  of $\F A$
  we actually have $e\mathrel{\R_A}e'$, which means that $e=e'e$ and $e'=ee'$,
  thus $e\mathrel{\R_X}e'$. This shows the claim that $e\in J_X(F\cap X^*)$.
  Hence, the above conclusion that $u\mathrel{\R_X}e$
  yields $J_A(F)\cap \F X\subseteq J_X(F\cap X^*)$,
  finishing the proof of~\ref{item:J-minimal-elements-are-almost-in-the-same-J-class-1}.

  For showing~\ref{item:J-minimal-elements-are-almost-in-the-same-J-class-3},
  it suffices to consider $\mathcal K=\mathcal R$, as case $\mathcal K=\mathcal L$ is symmetric and
  case $\mathcal K=\mathcal H$ follows straightforwardly from the other two cases.
  Since $\R_X\subseteq \R_A$, it is clear that every $\R_X$-class contained
  in $J_X(F\cap X^*)$ is contained in an intersection $V\cap\F X$
  for some $\R_A$-class $V$, with $V\subseteq J_A(F)$,
  as we saw that $J_X(F\cap X^*)\subseteq J_A(F)$.
  Conversely, let $V$ be any $\R_A$-class contained in
  $J_A(F)$ such that $V\cap\F X\neq\emptyset$.
  Note that $V\cap\F X\subseteq J_X(F\cap X^*)$, since we showed
  $J_A(F)\cap\F X\subseteq J_X(F\cap X^*)$.
  Let $v,w\in V\cap\F X$, and take $s,t\in \F A$ such that $v=ws$ and $w=vt$.
  Because $X$ is bifix, we actually have
  $s,t\in\F X$, by Proposition~\ref{p:right-unitary},
  whence $v\mathrel{\R_X}w$. Therefore, 
  $V\cap\F X$ is an $\R_X$-class contained in $J_X(F\cap X^*)$.
  This concludes the proof of~\ref{item:J-minimal-elements-are-almost-in-the-same-J-class-3}.

  Since every idempotent in $J_A(F)$ belongs to $\F X$ by Proposition~\ref{p:idempotents-are-in-the-closure-of-Xast},
  Property~\ref{item:J-minimal-elements-are-almost-in-the-same-J-class-3b}
  follows immediately from Property~\ref{item:J-minimal-elements-are-almost-in-the-same-J-class-3}.

  Finally, suppose that $xy\Rightarrow yX^*x\cap F\subseteq X^*$.
  We claim that
  \begin{equation}\label{eq:J-minimal-elements-are-almost-in-the-same-J-class-1}
   xy\in \overline{X^*}\Rightarrow y\F X x\cap \overline{F}\subseteq \F X.
 \end{equation}
 Assuming that $xy\in \overline{X^*}$, let $u\in y\F X x\cap\overline{F}$.
 Take $v\in\F X$ such that $u=yvx$
 and a sequence $(u_n)_n$ of elements of $F$
 converging to $u$. By Lemma~\ref{l:refine-factorization}, for every $n$ there are factorizations
 $u_n=y_nv_nx_n$ such that $\lim y_n=y$, $\lim v_n=v$ and $\lim x_n=x$.
 Since $\F X$ is a clopen subset of $\F A$, we may as well suppose that the words $v_n$ and $x_ny_n$ belong to $X^*$ for every $n$.
 By hypothesis, $x_ny_n\in X^*$
 implies $y_nX^*x_n\cap F\subseteq X^*$, whence $y_nv_nx_n\in X^*$ for every $n$. This shows that $u=yvx\in\F X$, establishing the claim.
 
  Let $e,f$ be idempotents in $J_X(F\cap X^*)$, and let $H_e,H_f$ be the maximal subgroups of $\F A$ respectively containing $e,f$.
  We already saw that $J_X(F\cap X^*)\subseteq J_A(F)\cap\F X$. Since $F$ is recurrent, $J_A(F)$ is a regular $\J_A$-class.
  Therefore, there are $x,y\in J_A(F)$ such that $e=xy$ and $f=yx$. Moreover, the mappings
  $\varphi\colon H_e\to H_f$ and $\psi\colon H_f\to H_e$ such that $\varphi(g)=ygx$ and $\psi(h)=xhy$, for every $g\in H_e$ and $h\in H_f$,
  are mutually inverse continuous isomorphisms. In view of~\ref{item:J-minimal-elements-are-almost-in-the-same-J-class-3b} we only need to check that
  $\varphi(H_e\cap \F X)=H_f\cap \F X$.
  Since $xy=e\in\F X$ and $J_A(F)\subseteq\overline{F}$, it follows from~\eqref{eq:J-minimal-elements-are-almost-in-the-same-J-class-1}
  that $\varphi(H_e\cap \F X)\subseteq H_f\cap \F X$.
  Similarly, the inclusion $\psi(H_f\cap \F X)\subseteq H_e\cap \F X$ also holds.
  Hence, $\varphi(H_e\cap \F X)= H_f\cap \F X$ indeed holds, concluding the proof of~\ref{item:J-minimal-elements-are-almost-in-the-same-J-class-4}.
\end{proof}

\section{Decoding of uniformly recurrent languages by finite codes}

For the sake of the proof of Theorem~\ref{t:charged-decoding-uniformly-recurrent},
we establish in this section some results
concerning the decoding of uniformly recurrent languages by arbitrary finite codes, not necessarily bifix.

\begin{Prop}\label{p:decoding-of-a-J-maximal-idempotent}
  Let $F$ be a uniformly recurrent language over the alphabet $A$. Let~$X$ be any finite code
  contained in $F$.
  For every $e\in E_X(F)$, the set $\Fin_X(e)$
  is a uniformly recurrent language over the alphabet $X$.
\end{Prop}

\begin{proof}
  Let $e\in E_X(F)$.
  Suppose that $u\in\F X$ is an infinite pseudoword
  such that $e\leq_{\J_X}u$.
  Take elements $x$ and $y$ of $\F X$ such
  that $e=xuy$.
  Then the pseudoword $f=(u\cdot y x)^2\in\F X$
  is an idempotent such that~
  $e\mathrel{\J_X} f$.
  Since $u$ is an infinite factor of $e$ in $\F A$,
  and $F$ is uniformly recurrent over $A$,
  we obtain from Theorem~\ref{t:jf-minimal-case}
  that $u\mathrel{\J_A}e$.
  As $\F A$ is stable, it follows that $u\mathrel{\R_A}f$,
  whence
  \begin{equation*}
    u=f\cdot u=uyxuyx\cdot u=
    \underset{\in\overline{X^*}}{\underbrace{uy}}
    \cdot
    e
    \cdot
    \underset{\in\overline{X^*}}{\underbrace{xu}}.
  \end{equation*}
  Therefore $e\mathrel{\J_X}u$, showing that $e$ is a $\J_X$-maximal
  infinite pseudoword of $\overline{X^*}$.
  This shows that there is a minimal subshift $\Cl S\subseteq X^\ZZ$
  such that $e\in J_X(\Block(\Cl S))$, by Theorem~\ref{t:jf-minimal-case}.
  As $\overline{\Block(\Cl S)}$ is a factorial subset of $\F X$,
  all factors of $J_X(\Block(\Cl S))$ in $X^*$ belong to $\Block(\Cl S)$,
  whence $\Fin_X(e)\subseteq \Block(\Cl S)$.
  Since $\Fin_X(e)$ is a subshift language (cf.~Lemma~\ref{l:finite-factors-factorial-prolongable}), we deduce from the minimality
  of $\Cl S$ that $\Fin_X(e)=\Block(\Cl S)$,
  thus showing that $\Fin_X(e)$ is uniformly recurrent.
\end{proof}

\begin{Cor}\label{c:recurrence-of-fx-is-enough}
  Let $F$ be a uniformly recurrent language over the alphabet $A$.
  Let~$X$ be any finite code contained in~$F$.
  If $F_X$ is recurrent over the alphabet $X$, then it is uniformly recurrent over $X$.
\end{Cor}

\begin{proof}
  As seen in Proposition~\ref{p:decoding-is-a-union},
  the set $F_X$ is the union of the finite collection $\{\Fin_X(e)\mid e\in E_X(F)\}$.
  By Proposition~\ref{p:decoding-of-a-J-maximal-idempotent},
  the elements of this collection are uniformly recurrent languages over $X$.
  
  It remains to observe
  that if a union of a finite collection of uniformly recurrent sets is recurrent, then it is actually uniformly recurrent. We give a ``profinite'' proof for
  this known fact. Suppose that $F=\bigcup_{i=1}^kF_i$ is a finite union
  of $k$ distinct  sets $F_i\subseteq X^*$ that are uniformly recurrently over $X$, and that $F$ is recurrent over $X$. For each $i$, let
  $u_i\in J(F_i)$. As $\overline{F}$ is recurrent,
  there are $x_1,\ldots,x_{k-1}\in\F X$ such that
  the pseudoword $w=u_1x_1u_2x_2\cdots u_{k-1}x_{k-1}u_k$ belongs to $\overline{F}$.
  Since $\overline{F}=\bigcup_{i=1}^k\overline{F_i}$,
  we then have $w\in\overline{F_j}$ for some $j$.
  Because $F_j$ is uniformly recurrent,
  we know that $w$ is an $\J_X$-maximal infinite pseudoword,
  and so $u_i\mathrel{\J_X}w$ for every $i\in\{1,\ldots,k\}$.
  But $J(F_i)=J(F_j)$ implies $F_i=F_j$, and so necessarily $k=1$.
\end{proof}

\begin{Cor}\label{c:maximal-bifix-uniformly-recurrent-is-union-of-uniformly-recurrent}
  If the complete bifix decoding
  of a uniformly recurrent language
  is recurrent, then it is uniformly recurrent.
\end{Cor}

\begin{proof}
  This
  is a special case of Corollary~\ref{c:recurrence-of-fx-is-enough}:
  for a uniformly recurrent language $F$ and a finite $F$-complete bifix code $X$, we have $F\cap X^*=F_X$ by Proposition~\ref{p:Fx-in-the-bifix-case}.
\end{proof}

\section{Proof of Theorem~\ref{t:charged-decoding-uniformly-recurrent}}
\label{sec:proof-theor-reft}

In this section we prove that every charged complete bifix decoding
  of a uniformly recurrent language is uniformly recurrent (Theorem~\ref{t:charged-decoding-uniformly-recurrent}).
The following fact is needed.
  
\begin{Lemma}\label{l:closure-star-z-int-f}
  Let $F$ be a factorial language of~$A^*$. If $Z$ is a rational language
  of~$A^*$,
  then the inclusion $\overline{Z^*}\cap \overline{F}\subseteq \overline{(Z\cap F)^*}$ holds in $\F A$.
\end{Lemma}

\begin{proof}
  Let $u\in \overline{Z^*}\cap \overline{F}$.
  Take a sequence $(u_n)_{n\in\NN}$ of elements of $F$
  converging to~$u$. Since $Z^*$ is rational, $\overline{Z^*}$ is open, and so there is $N\in\NN$ such that $n\geq N$ implies $u_n\in Z^*\cap F$.
  As $F$ is factorial, we have $Z^*\cap F=(Z\cap F)^*$,
  thus $u\in\overline {(Z\cap F)^*}$.
\end{proof}

In the next proof,  we
denote by $H_u$ the $\H$-class
of an element $u$ of a monoid $M$.

\begin{Prop}\label{p:unique-J-class-when-decoding}
  Let $F$ be a recurrent language over the alphabet $A$.
  If $X$ is a finite $F$-complete $F$-charged bifix code,
  then $F\cap X^*=\Fin_X(e)$
  for every $e\in E_A(F)$.
\end{Prop}

\begin{proof}
  Let $Z\subseteq A^+$ be an $F$-charged rational complete bifix code for which we have $X=Z\cap F$.
  Let $e,f\in E_A(F)$. Recall that $e,f\in\overline{X^*}$
  by Proposition~\ref{p:idempotents-are-in-the-closure-of-Xast}.
  By Proposition~\ref{p:Fx-in-the-bifix-case},
  it suffices
  to show that $e\mathrel{\J_X}f$.

  As $e\mathrel{\J_A}f$, we may take $x,y\in\F A$ such that $e=xfy$,
  with $x=exf$ and $y=fye$.
  Since $Z$ is $F$-charged, the inclusion
  $\hat\eta_{Z^*}(J_A(F))$ is contained in the minimum ideal
  of the syntactic monoid $M_{Z^*}$, and so
  \begin{equation*}
    \hat\eta_{Z^*}(x)^\omega\mathrel{\H}
    \hat\eta_{Z^*}(x)\mathrel{\R}\hat\eta_{Z^*}(e).
  \end{equation*}
  Therefore, since $\hat\eta_{Z^*}(e)\cdot \hat\eta_{Z^*}(x)=\hat\eta_{Z^*}(x)$,
  we know from Green's Lemma (cf.~\cite[Lemma A.3.1]{Rhodes&Steinberg:2009qt}) that there is $s\in H_{\hat\eta_{Z^*}(e)}$
  such that $s\cdot \hat\eta_{Z^*}(x)=\hat\eta_{Z^*}(x)^\omega$.
  Because $Z$ is $F$-charged,
  we have $\hat\eta_{Z^*}(H_e)=H_{\hat\eta_{Z^*}(e)}$
  and we may take some $u\in H_e$ such that $s=\hat\eta_{Z^*}(u)$.
    Since $u=ue=uxfy\in J_A(F)$, we have $ux\in\overline{F}$. As $\hat\eta_{Z^*}(ux)=\hat\eta_{Z^*}(x)^\omega$
  is idempotent, applying Proposition~\ref{p:idempotents-are-in-the-closure-of-Xast} we conclude that $ux\in \overline{Z^*}$.

  Symmetrically, there is $v\in H_e$  such that
  $yv\in\overline{Z^*}$. 

  Consider the pseudoword $w=ux\cdot f\cdot yv$.
  Because $xfy=e$ and $u,v\in H_e$, we have $w=uv\in H_e$.
  In particular, the equality
  \begin{equation}\label{eq:unique-J-class-when-decoding-0}
    (ux\cdot f\cdot yv)^{\omega-1}\cdot w=e
  \end{equation}
  holds.
  Again by Proposition~\ref{p:idempotents-are-in-the-closure-of-Xast},
  the idempotents $e,f$ belong to $\overline{Z^*}$,
  and since $ux$ and $yv$
  also belong to $\overline{Z^*}$,
  we conclude by the second implication in Proposition~\ref{p:right-unitary}
  that $w\in\overline{Z^*}$.
  And since $\overline{F}$ is factorial,
  the pseudowords $ux$, $yv$, $f$, $e$ and $w$
  belong to $\overline{F}$,
  and so they belong to
  $\overline{X^*}$ by Lemma~\ref{l:closure-star-z-int-f}.
  Therefore, the equality~\eqref{eq:unique-J-class-when-decoding-0}
  yields $e\leq_{\J_X}f$. By symmetry, we also have $f\leq_{\J_X}e$, concluding the proof.  
\end{proof}

\begin{proof}[Proof of Theorem~\ref{t:charged-decoding-uniformly-recurrent}]
  Combining Propositions~\ref{p:unique-J-class-when-decoding}
  and~\ref{p:decoding-of-a-J-maximal-idempotent},
  we immediately obtain Theorem~\ref{t:charged-decoding-uniformly-recurrent}.
\end{proof}

  \section{Proof of Theorem~\ref{t:charged-decoding-recurrent}}
  \label{sec:proof-theor-reft:recurrent}
  
  We proceed to show that every charged complete bifix decoding
  of a recurrent language is recurrent (Theorem~\ref{t:charged-decoding-recurrent}).
  
\begin{proof}[Proof of Theorem~\ref{t:charged-decoding-recurrent}]
  Let $F$ be a recurrent language over the alphabet $A$,
  and let $Z$ be an $F$-charged rational complete bifix code.
  Consider the code $X=Z\cap F$, and suppose that $X$ is finite.
  We want to show that $F\cap X^*$ is a recurrent language over the alphabet $X$.
  Let $e\in E_A(F)$.
  By Proposition~\ref{p:unique-J-class-when-decoding},
  what we want to show is that $\Fin_X(e)$ is recurrent.
  Let $u,v\in \Fin_X(e)$.
    Take $x,z,v,t\in \overline{X^*}$ such that $e=xuy=zvt$.
  Take a sequence $(w_n)_{n\in\NN}$ of elements of $F$
  converging to $e$.
  Applying Lemma~\ref{l:refine-factorization}
  to the equality
  \begin{equation*}
    e=x\cdot u\cdot y\cdot
    z\cdot v\cdot t
  \end{equation*}
  in the free profinite monoid $\F X$,
  we conclude that, for every $n\in\NN$, there are factorizations
  \begin{equation*}
    w_n=x_n\cdot u_n\cdot y_n\cdot z_n\cdot v_n\cdot t_n
  \end{equation*}
  in $X^*$ such that, in $\F X$, we have
  \begin{equation*}
    x_n\to x,\quad u_n\to u,\quad y_n\to y,\quad z_n\to z,\quad v_n\to v,
    \quad t_n\to t.
  \end{equation*}
  Since $u,v\in X^*$, by taking subsequences we may as well suppose that $u_n=u$ and $v_n=v$ for every $n\in\NN$.
  Because $uy_nz_nv$ is a factor of $w_n\in F$, and $F$
  is factorial, we know that $uy_nz_nv$ belongs to $F\cap X^*$.
  By Lemma~\ref{l:prefix-of-some-idempotent-in-JF}, there is an idempotent
  $f\in J_A(F)$
  of the form
  \begin{equation}\label{eq:decoding-of-a-J-minimal-recurrent-idempotent-0}
  f=uy_nz_nvr,  
  \end{equation}
  for some $r\in\F A$.
  Recall that $f\in\overline{X^*}$
  by Proposition~\ref{p:idempotents-are-in-the-closure-of-Xast}.
  As $uy_nz_nv\in X^*$
  and $X$ is a prefix code,
  applying Proposition~\ref{p:right-unitary}
  we conclude that $r\in\overline{X^*}$.
  Therefore, from~\eqref{eq:decoding-of-a-J-minimal-recurrent-idempotent-0}
  we deduce that $uy_nz_nv\in\Fin_X(f)$.
  Since by Proposition~\ref{p:unique-J-class-when-decoding}
  the equality $\Fin_X(f)=\Fin_X(e)$
  holds, we have thus established
  that $\Fin_X(e)$ is a recurrent language over the alphabet $X$.
\end{proof}

\section{The Sch\"utzenberger group of an irreducible subshift is invariant under eventual conjugacy}
\label{sec:eventual-conjugacy}

In the category of symbolic dynamical systems, a morphism between a subshift~$\Cl S$ of $A^{\ZZ}$ and a subshift
$\Cl T$ of $B^{\ZZ}$ is a continuous mapping $\varphi\colon \Cl S\to\Cl T$
such that $\sigma_B(\varphi(x))=\varphi(\sigma_A(x))$ for every $x\in\Cl S$.
If there is
a map $\Phi\colon A\to B$
such that $\varphi(x)=(\Phi(x_i))_{i\in\ZZ}$ for each $x\in \Cl S$,
then $\varphi$ is said to be a \emph{one-block code} with associated
\emph{letter-to-letter block map} $\Phi$.

An isomorphism of subshifts is called a \emph{conjugacy}, and isomorphic subshifts are \emph{conjugate}.
A \emph{one-block conjugacy} is a one-block code that is a conjugacy.
It turns out that conjugacy is the equivalence relation generated by one-block conjugacies. More precisely,
the following holds (cf.~\cite[Proposition~1.5.12]{Lind&Marcus:1996}).

\begin{Prop}\label{p:decomposition-of-morphisms}
If $\varphi\colon \Cl S\to\Cl T$ is a conjugacy,
then there is a subshift $\Cl R$ and one-block conjugacies
$\alpha\colon\Cl R\to\Cl S$ and $\beta\colon\Cl R\to\Cl T$
such that $\varphi=\beta\circ\alpha^{-1}$.
\end{Prop}

Two subshifts $\Cl S$ and $\Cl T$ are \emph{eventually conjugate} when there is a positive integer~$N$ such that the higher power shifts $\Cl S^n$ and $\Cl T^n$ are conjugate for every $n\geq N$.
If the subshifts $\Cl S$ and $\Cl T$ are conjugate, then $\Cl S$ and $\Cl T$ are eventually conjugate.

\subsection{The Sch\"utzenberger group of an irreducible subshift and its higher powers}

 Let $\Cl S$ be an irreducible subshift of $A^{\ZZ}$.
 Since the $\J$-class $J_A(\mathcal B(S))$ is regular, we may consider its maximal subgroups, and identify them
 as a single profinite group $G(\Cl S)$, which since~\cite{Almeida&ACosta:2013}
 is called the \emph{Sch\"utzenberger group} of $\Cl S$.
 In this section, we show that $G(\Cl S)$ is an invariant of eventual conjugacy, for any irreducible subshift~$\Cl S$. The invariance of $G(\Cl S)$ under conjugacy was first proved in~\cite{ACosta:2006}.

 We next generalize the definition of Sch\"utzenberger group of an irreducible subshift to (possibly not irreducible) higher powers of irreducible subshifts, as follows.
 Given an irreducible subshift $\Cl S$ of $A^{\ZZ}$, let $F=\Block(\Cl S)$.
 For every positive integer~$n$ and irreducible subshift $\Cl S\subseteq A^{\ZZ}$, the equality $\mathcal B(S^n)=F\cap (A^n)^*$ holds. Moreover, by Proposition~\ref{p:J-minimal-elements-are-almost-in-the-same-J-class}:
 \begin{itemize}
 \item the set $J(\Cl S^n)=J_{A^n}(F\cap (A^n)^*)$ of $\leq_{\Cl J_{A^n}}$-minimal elements of $\overline{F\cap (A^n)^*}$ is a union of regular $\J_{A^n}$-classes;
 \item any two maximal subgroups contained in $J(\Cl S^n)$ are isomorphic profinite groups.
 \end{itemize}
 We identify the maximal subgroups contained in $J(\Cl S^n)$
 as a single profinite group, which we call the \emph{Sch\"utzenberger group of $\Cl S^n$} and denote~$G(\Cl S^n)$. Note that if $\Cl S^n$ is irreducible (which happens if $n=1$),
 then we get the Sch\"utzenberger group of~$\Cl S^n$ as initially defined, and so we indeed have a consistent generalization.

 To show the invariance under eventual conjugacy of $G(\Cl S)$ when $\Cl S$
 is irreducible, we need to use not only the invariance
 under conjugacy of $G(\Cl S)$, but the following stronger result.

 \begin{Thm}\label{t:conjugacy-invariance}
   If $\Cl S$ and $\Cl T$ are conjugate irreducible subshifts, then $G(\Cl S^n)$ and $G(\Cl T^n)$ are
   isomorphic profinite groups, for every $n\geq 1$.
 \end{Thm}

 Theorem~\ref{t:conjugacy-invariance} is also from~\cite{ACosta:2006},
 where one finds a more general result concerning the restriction of the
 quasi-order $\leq_{\J}$
 to $\overline{\Block(\Cl S)}$:
 if $\Cl S$ is a (possibly not irreducible) subshift of $A^\ZZ$, then we have invariance under conjugacy of
 the set of isomorphism classes of the maximal subgroups containing idempotents that are $\leq_{\J_A}$-minimal among the idempotents in $\overline{\Block(\Cl S)}$ (cf.~\cite[Theorem 3.11]{ACosta:2006}).
 
 The free procyclic group on one generator $\widehat{\ZZ}$ is the inverse limit $\varprojlim \ZZ/{\ZZ_n}$,
 in which $\ZZ$ embeds densely (cf.~\cite{Ribes&Zalesskii:2010}).
For each alphabet $A$ and pseudoword $u$ in $\F A$,  the \emph{procyclic image}
of $u$ (with respect to $A$) is the image of $u$ in $\widehat{\ZZ}$ by the unique continuous
homomorphism $\ell^A\colon\F A\to \widehat{\ZZ}$ such that $\ell^A(a)=1$ for every $a\in A$.
Note that if $u\in A^*$, then $\ell^A(u)$ is just the length of $u$.
Therefore, the notion of procyclic image of a pseudoword is a generalization of the notion of word length.

The proof in \cite{ACosta:2006} that $G(\Cl S)$ is a conjugacy invariant of irreducible subshifts
gives for free that the procyclic image of the elements of $G(\Cl S)$ is preserved by conjugacy, as seen next.

\begin{Thm}\label{t:profinite-length-preserved}
  Suppose that $\Cl S\subseteq A^{\ZZ}$ and $\Cl T\subseteq B^{\ZZ}$ are conjugate irreducible subshifts.
  If $H$ is a maximal subgroup of $J(\Cl S)$ and $K$ is a maximal subgroup of~$J(\Cl T)$,
  then there is an isomorphism of profinite groups $\psi\colon H\to K$
  that preserves procyclic images, that is to say, such that
  \begin{equation*}
    \ell^B(\psi(h))=\ell^A(h)
  \end{equation*}
  for every $h\in H$.
\end{Thm}

\begin{proof}
  It suffices to assume that there is a one-block conjugacy $\varphi\colon \Cl S\to\Cl T$, by Proposition~\ref{p:decomposition-of-morphisms}.
  Let  $\Phi\colon A\to B$ be the associated letter-to-letter block map.
  For every maximal subgroup $H$ contained in $J(\Cl S)$,
  the unique continuous homomorphism  $\widehat\Phi\colon \F A\to\F B$ extending $\Phi$
  restricts to a continuous isomorphism from $H\subseteq J(\Cl S)$ onto a maximal subgroup $N\subseteq J(\Cl T)$, a property of $\widehat\Phi$
  shown in \cite{ACosta:2006}. (More precisely, applying Proposition 3.10 from \cite{ACosta:2006} we see that $N=\widehat\Phi(H)$ is a maximal subgroup,
  and Theorem 3.11 from \cite{ACosta:2006} guarantees that $N\subseteq J(\Cl T)$.)
  Note that $\ell^B(\widehat\Phi(u))=\ell^A(u)$ for every $u\in \F A$, since that clearly holds when $u\in A^*$.
  
  Finally, let $K$ be any maximal subgroup of $\F B$ contained in $J(\Cl T)$.
  Then there are $x,y\in J(\Cl T)$ such that $xy$ is the idempotent in $N$, $yx$  is the idempotent in $K$, and the mapping
  $\lambda\colon N\to K$ given by $\lambda(g)=ygx$ is a continuous isomorphism.
  As $xy$ is idempotent, we have $\ell^B(x)=-\ell^B(y)$.
  The fact that $\widehat{\ZZ}$ is Abelian furthermore yields $\ell^B(\lambda(g))=\ell^B(g)$
  for every $g\in N$. Therefore, the homomorphism $\psi\colon H\to K$ given by
   $\psi=\lambda\circ(\widehat\Phi|_H)$ to $H$ is a continuous isomorphism such that
    $\ell^B(\psi(h))=\ell^A(h)$ for every $h\in H$.
  \end{proof}

   \subsection{Proof of the invariance of $G(\Cl S)$ under eventual conjugacy}

For each positive integer $n$, denote $\rho_n$ the unique continuous homomorphism
$\widehat{\ZZ}\to \ZZ/{n\ZZ}$ such that $\rho_n(1)$ is the class modulo $n$ of~$1$. The composition $\rho_n\circ\ell^A$ is the homomorphism
$\ell_n^A\colon \F A\to\ZZ/{n\ZZ}$ introduced before Corollary~\ref{c:the-image-is-the-cyclic-group-modulo-n}.

\begin{Rmk}
  The length of a word $u$ over the alphabet $A$ is a multiple
  of $n$ if and only if $\ell_n^A(u)=0$. Therefore, by continuity of the homomorphism $\ell_n^A\colon\F A\to \ZZ/{n\ZZ}$, the equality
  \begin{equation*}
    (\ell_n^A)^{-1}(0)=\overline{(A^n)^*}
  \end{equation*}
  holds.
\end{Rmk}

\begin{Lemma}\label{l:profinite-length-modulo-change}
  If $n$ and $m$ are relatively prime positive integers, then
  for every $u\in \F A$ such that $u\in\overline{(A^m)^*}$, the equivalence
  \begin{equation*}
    \ell_n^{A}(u)=0\Leftrightarrow \ell_n^{A^m}(u)=0
  \end{equation*}
  holds.
\end{Lemma}

\begin{proof}
  Because $(A^n)^*$ and $(A^m)^*$ are rational, we have
  \begin{equation*}
    \overline{(A^n)^*}\cap \overline{(A^m)^*}=\overline{(A^n)^*\cap (A^m)^*}.
  \end{equation*}
  Since $n$ and $m$ are relatively prime,
  we also have
  \begin{equation*}
   (A^n)^*\cap (A^m)^*=(A^{nm})^*.
 \end{equation*}
 Notice also that $(A^{nm})^*=((A^{m})^n)^*$.
  Therefore, for every $u\in \overline{(A^m)^*}$, the chain of equivalences
  \begin{equation*}
    \ell_n^{A}(u)=0
    \Leftrightarrow u\in \overline{(A^n)^*}\cap \overline{(A^m)^*}
    \Leftrightarrow u\in \overline{((A^{m})^n)^*}
    \Leftrightarrow \ell_n^{A^m}(u)=0
  \end{equation*}
  holds.
\end{proof}

 \begin{Lemma}\label{l:the-maximal-subgroup-of-hygher-power-collapses}
  Let $\Cl S$ be an irreducible subshift of $A^{\ZZ}$.
  Let $n$ be a positive integer.
  Suppose that the maximal subgroup $H$ of $J(\Cl S)$
  satisfies $\ell_n^A(H)=0$.
  Then $H$ is a maximal subgroup of $J(\Cl S^n)$.
\end{Lemma}

\begin{proof}
  Let $X=\Block(\Cl S)\cap A^n$.
  By Proposition~\ref{p:J-minimal-elements-are-almost-in-the-same-J-class},
  the intersection $H\cap\F X$ is a maximal subgroup of $J(\Cl S^n)\subseteq\overline{\Block(\Cl S)}$.
  On the other hand, the equality $\ell_n^A(H)=0$ gives $H\subseteq \overline{(A^n)^*}\cap\overline{\Block(\Cl S)}$.
  Since $\overline{(A^n)^*}\cap\overline{\Block(\Cl S)}\subseteq\overline{(A^n)^*\cap \Block(\Cl S)}=\F X$, we are done.
\end{proof}
   
For each positive integer $k$ and closed subgroup $H$ of $\F A$, we denote by $H(k)$ the  intersection $H\cap\overline{(A^k)^*}$.
In particular, we have $H(1)=H$.

\begin{Rmk}\label{rmk:intersection-of-Hm-Hn}
  Note that we always have $H(nm)\subseteq H(n)\cap H(m)$. Moreover, the equality
  $H(nm)=H(n)\cap H(m)$ holds whenever $m,n$ are relatively prime.
\end{Rmk}

Note that $H(k)=H\cap (\ell^A_k)^{-1}(0)$ is a closed normal subgroup of $H$.
   Hence,  we may consider for all $n,m\geq 1$ the
   profinite semidirect product $H(n)\ltimes H(m)$ induced by the left  action in which ${}^xu=xux^{-1}$ for all $x\in H(m)$ and $u\in H(n)$.

   \begin{Rmk}\label{r:closed-normal-subgroup}
     For every closed normal subgroup $H$ of $\F A$,
     the semidirect product $H(m)\ltimes H(n)$ is a closed normal subgroup of
     $H\ltimes H$: it is the kernel of the continuous homomorphism
     $H\ltimes H\to \ZZ/{m\ZZ}\times \ZZ/{n\ZZ}$
     mapping each $(u,x)\in H\ltimes H$ to $(\ell_m^A(u),\ell_n^A(x))$.
   \end{Rmk}

   The following lemma is crucial for our proof that the Sch\"utzenberger group of an irreducible subshift is invariant under eventual conjugacy.
   The notation $\cong$ stands for isomorphism of compact groups.

   \begin{Lemma}\label{l:semidirect-decomposition-of-the-maximal-subgroup-H}
     Let $\Cl S$ be an irreducible subshift of $A^{\ZZ}$, and let $H$ be a maximal subgroup
     contained in $J(\Cl S)$.
     Let $m,n$ be relatively prime positive integers such that $\ell_m(H)=\mathbb Z/{m\ZZ}$
     and $\ell_n(H)=\mathbb Z/{n\ZZ}$.
     Then we have
     \begin{equation*}
       H\cong (H(m)\ltimes H(n))/(H(mn)\ltimes H(mn)).
     \end{equation*}
     More precisely, the mapping
     \begin{equation*}
       \begin{array}{rccc}
          \varphi\colon &H(m)\ltimes H(n) & \to & H\\
         &(u,x) & \mapsto & ux
       \end{array}
     \end{equation*}
     is a continuous onto homomorphism of profinite groups,
     and the kernel of $\varphi$ is the profinite group $H(mn)\ltimes H(mn)$.
   \end{Lemma}

   \begin{proof}
     The map $\varphi$ is clearly continuous, and it is also straightforward to see that it is a homomorphism: $\varphi((u,x)(v,y))=\varphi(uxvx^{-1},xy)=uxvy=\varphi(u,x)\varphi(v,y)$.
    
     We proceed to show that $\varphi$ is onto. Let $h\in H$.
     Since $m,n$ are relatively prime, there is a
     natural number $k_0$ such that, for every natural $k>k_0$, there are natural numbers
     $k_1,k_2$ such that $k=k_1m+k_2n$. Therefore, every
     word of $A^*$ of length greater than $k_0$ belongs to $(A^m)^*\cdot (A^n)^*$,
     and so every infinite pseudoword of $\F A$ belongs
     to $\overline{(A^m)^*}\cdot \overline{(A^n)^*}$.
     In particular, if $h\in H$, then there are pseudowords $u\in \overline{(A^m)^*}$
     and $x\in \overline{(A^n)^*}$ such that $h=ux$.
     Let $e$ be the idempotent in $H$. Then we have $h=euxe$,
     whence $eu\mathrel{\R_A}h\mathrel{\L_A}xe$.
     Since $J(\Cl S)$ is regular, we may take an idempotent $f\in J(\Cl S)$ such that $eu=euf$.

     By the Chinese Remainder Theorem, we have $\ZZ/{mn\ZZ}=\ZZ/{m\ZZ}\times\ZZ/{n\ZZ}$
     and 
     $\ell_{mn}(w)=(\ell_m(w),\ell_n(w))$ for every $w\in\F A$.
     Hence, the equalities $\ell_m(H)=\ZZ/{m\ZZ}$ and $\ell_n(H)=\ZZ/{n\ZZ}$ entail
     $\ell_{mn}(H)=\ZZ/{mn\ZZ}$. In particular, $\Cl S^{mn}$ is irreducible
     by Corollary~\ref{c:the-image-is-the-cyclic-group-modulo-n}.
     Therefore, letting $X=\Block(\Cl S)\cap A^{mn}$,
     we may consider the $\J_X$-class $J(\Cl S^{mn})$.
     Note that $\ell_{mn}(e)=\ell_{mn}(f)=0$, and thus $e,f$ belong to $\overline{\Cl B(\Cl S)}\cap \F X$,
     that is, $e,f\in \overline{\Cl B(\Cl S^{mn})}$.
     It follows from Proposition~\ref{p:J-minimal-elements-are-almost-in-the-same-J-class} that
     $e,f$ are both in the $\J_X$-class $J(\Cl S^{mn})$. Hence,
     there are $z,t\in\F X$ such that $f=fzetf$.

     Consider the pseudowords $u'=eufze$ and $x'=etfxe$.
     We have
     \begin{equation*}
       u'x'=eu\cdot fzetf\cdot xe=euf\cdot xe=euxe=h.
     \end{equation*}
     Note that $u',x'$ are $\H$-equivalent to $e$.
     As $\ell_{mn}(\F X)=0$ and $\ell_m(u)=0$, we have
     $\ell_m(u')=\ell_m(eufze)=0$. Similarly, we also have $\ell_n(x')=0$.
     Therefore, the pair $(u',x')$ belongs to $H(m)\ltimes H(n)$ and $\varphi(u',x')=h$, thus showing
     that $\varphi$ is onto.

     It remains to check that $\ker\varphi=H(mn)\ltimes H(mn)$.
     Take $(u,x)\in H(m)\ltimes H(n)$ for which $\varphi(u,x)=e$.
     Observe that $0=\ell_n(ux)=\ell_n(u)+\ell_n(x)=\ell_n(u)$,
     which together with $\ell_m(u)=0$ gives $\ell_{mn}(u)=0$,
     that is, $u\in H(mn)$. Similarly, we also have $x\in H(mn)$. We conclude that $(u,x)$ belongs to $H(mn)\ltimes H(mn)$.
   \end{proof}

   We are now ready to show the main result of this section.
   
\begin{Thm}\label{t:invariance-under-eventual-conjugacy}
  If $\Cl S_1$ and $\Cl S_2$ are eventually conjugate irreducible subshifts, then $G(\Cl S_1)$ and $G(\Cl S_2)$ are
   isomorphic profinite groups.
 \end{Thm}

 \begin{proof}
   We may consider an alphabet $A$ such that
   both $\Cl S_1$ and $\Cl S_2$ are subshifts of $A^\ZZ$ (let $A$ for example be the set of letters appearing in some element of $\Cl S_1$ or $\Cl S_2$).
   
   Since $\Cl S_1$ and $\Cl S_2$ are eventually conjugate,
  we may fix a positive integer $N$ such that  $\Cl S_1^n$ and $\Cl S_2^n$ are conjugate for every $n\geq N$.
   
   For each $i\in\{1,2\}$, take a maximal subgroup $H_i$ of $\F {A}$ contained in $J(\Cl S_i)$.
    Consider the set
     $Q_i$ of prime numbers greater than $N$ and such that $\ell_p^A(H_i)=0$.

   \begin{Lemma}\label{l:qi-finite}
     The sets $Q_2\setminus Q_1$ and $Q_1\setminus Q_2$ have at most one element.
   \end{Lemma}

   \begin{proof}
     Suppose the lemma is false. Without loss of generality,
     suppose that $Q_2\setminus Q_1$ has two distinct elements
     $p,q$.
     Since $\ell_p^A(H_1)\neq 0$ and $p$ is prime, we must have $\ell_p^A(H_1)=\ZZ/{p\ZZ}$. By Corollary~\ref{c:the-image-is-the-cyclic-group-modulo-n},
     the subshift $S_1^p$ is irreducible. Furthermore, the intersection
     $H_1\cap \overline{(A^p)^*}$ is a maximal subgroup of the $\J_{A^p}$-class $J(S_1^p)$ by Proposition~\ref{p:J-minimal-elements-are-almost-in-the-same-J-class}. Because $S_1^p$ and $S_2^p$ are conjugate (as $p>N$),
     the subshift $S_2^p$ is also irreducible. Moreover, since $\ell_p^A(H_2)=0$, it follows from  Lemma~\ref{l:the-maximal-subgroup-of-hygher-power-collapses}
     that $H_2$ is a maximal subgroup of the $\J_{A^p}$-class $J(\Cl S_2^p)$.
     Therefore, by Theorem~\ref{t:profinite-length-preserved}, there is a continuous isomorphism $\psi\colon H_2\to H_1\cap \overline{(A^p)^*}$
     such that $\ell^{A^p}\circ \psi=\ell^{A^p}|_{H_2}$.
     Hence, applying Lemma~\ref{l:profinite-length-modulo-change}, we have the chain of equivalences
     \begin{align*}
       \ell_q^{A}(H_2)=0
       &\Leftrightarrow
         \ell_q^{A^p}(H_2)=0
         \qquad\text{(because $H_2\subseteq \overline{(A^p)^*}$, as $p\in Q_2$)}
         \\
       &\Leftrightarrow
       \ell_q^{A^p}(H_1\cap \overline{(A^p)^*})=0\qquad\text{(as $\ell^{A^p}\circ \psi=\ell^{A^p}|_{H_2}$)}\\
       &\Leftrightarrow
       H_1\cap \overline{(A^p)^*}\subseteq \overline{((A^p)^q)^*}.
     \end{align*}
     Since $q\in Q_2$, we indeed have $\ell_q^A(H_2)=0$, whence
     \begin{equation}\label{eq:claim-1}
     H_1\cap \overline{(A^p)^*}\subseteq H_1\cap \overline{(A^q)^*}.  
     \end{equation}
     As $q\notin Q_1$, we may take some $h\in H_1$ such that $\ell_q^A(h)\neq 0$.
     Because $p\neq q$, we must have $p\ell_q^A(h)\neq 0$, that is, $h^p\notin \overline{(A^q)^*}$.
     But this contradicts~\eqref{eq:claim-1}.

     To avoid a contradiction, each of the sets $Q_2\setminus Q_1$ and $Q_1\setminus Q_2$ must have at most one element.
   \end{proof}

   We proceed with the proof of the theorem.
   
   Suppose first that at least one of the sets $Q_1$ and $Q_2$ is infinite.
   Then $Q_1\cap Q_2$ is infinite, by Lemma~~\eqref{l:qi-finite}.
   Let $p\in Q_1\cap Q_2$. As $p> N$, the subshifts $\Cl S_1^p$ and $\Cl S_2^p$ are conjugate.
   Since $\ell_p(H_i)=0$, it follows from Lemma~\ref{l:the-maximal-subgroup-of-hygher-power-collapses} that $H_i$ is a maximal subgroup of
   the $\J_{A^p}$-class of $J(\Cl S_i^p)$, for each $i\in \{1,2\}$.
   Applying Theorem~\ref{t:conjugacy-invariance},
   we conclude that $H_1\cong H_2$, settling the theorem for this case.

   We now consider the remaining case where $Q_1$ and $Q_2$ are finite. Then there are distinct prime numbers $p,q$
   greater than $N$ such that $p,q\notin Q_1\cup Q_2$.
   Let $r\in\{p,q\}$.
   Since $r$ is prime, we must have  $\ell_r(H_i)=\ZZ/{r\ZZ}$, for each $i\in \{1,2\}$.
   By Corollary~\ref{c:the-image-is-the-cyclic-group-modulo-n},
   the subshift $\Cl S_i^r$ is irreducible and, in view of Proposition~\ref{p:J-minimal-elements-are-almost-in-the-same-J-class}, the profinite group
   $H_i(r)$
   is a maximal subgroup of  the $\J_{A^r}$-class $J(\Cl S_i^r)$.
   Therefore, since $S_1^r$ is conjugate to $S_2^r$ (as $r>N$),
   it follows from Theorem~\ref{t:profinite-length-preserved}
   that there is a continuous isomorphism
   $\psi_{r}\colon H_1(r)\to H_2(r)$
   such that
   \begin{equation}\label{eq:qi-finite-0}
     \ell^{A^r}\circ \psi_r=\ell^{A^r}|_{H_1(r)}\qquad\text{(for each $r\in\{p,q\}$)}.
   \end{equation}
   Since $p,q$ are distinct primes,
   for each $i\in \{1,2\}$ we have
   \begin{equation*}
        H_i(pq)=H_i(p)\cap H_i(q)=H_i(p)\cap\overline{(A^q)^*}.
    \end{equation*}      
   Therefore, for every $h\in H_1(p)$, we have
   \begin{align*}
     \psi_p(h)\in H_2(pq)
     &\Leftrightarrow
       \psi_p(h)\in H_2(q)=H\cap \overline{(A^q)^*}.
     \\
     &\Leftrightarrow
       \ell_q^{A}(\psi_p(h))=0
     \\
     &\Leftrightarrow
       \ell_q^{A^p}(\psi_p(h))=0\qquad\text{(by~Lemma~\ref{l:profinite-length-modulo-change}, as $\psi_p(h)\in \overline{(A^p)^*}$)}
       \\
     &\Leftrightarrow
       \ell_q^{A^p}(h)=0\qquad\text{(by~\eqref{eq:qi-finite-0})}
     \\
      &\Leftrightarrow
       \ell_q^{A}(h)=0\qquad\text{(by~Lemma~\ref{l:profinite-length-modulo-change}, as $h\in \overline{(A^p)^*}$)}
     \\
     &\Leftrightarrow
       h\in H_1(p)\cap\overline{(A^q)^*}
    \\
     &\Leftrightarrow
     h\in H_1(pq).
   \end{align*}
   Similarly, for every $h\in H_1(q)$, the equivalence
   \begin{equation*}
     \psi_q(h)\in H_2(pq)\Leftrightarrow h\in H_1(pq)
   \end{equation*}
   holds. Therefore, the isomorphism
   $\psi\colon H_1(p)\ltimes H_1(q)\to H_2(p)\ltimes H_2(q)$
   given by the formula $\psi(u,x)=(\psi_p(u),\psi_q(x))$ satisfies 
   \begin{equation*}
     \psi(u,x)\in H_2(pq)\ltimes H_2(pq)\Leftrightarrow
     (u,x)\in H_1(pq)\ltimes H_1(pq).
   \end{equation*}
   Since $H_2(pq)\ltimes H_2(pq)$ is a closed
   normal subgroup of $H_2(p)\ltimes H_2(q)$ (cf.~Remark~\ref{r:closed-normal-subgroup}),
   it follows that
   \begin{equation*}
     (H_1(p)\ltimes H_1(q))/(H_1(pq)\ltimes H_1(pq))
     \cong
     (H_2(p)\ltimes H_2(q))/(H_2(pq)\ltimes H_2(pq)).
   \end{equation*}
   Applying Lemma~\ref{l:semidirect-decomposition-of-the-maximal-subgroup-H},
   we conclude that $H_1\cong H_2$.   
 \end{proof}

 \subsection{A digression on relatively free profinite monoids}

 We finish by explaining how to generalize Theorem~\ref{t:invariance-under-eventual-conjugacy}
 to  \emph{relatively} free profinite monoids over certain pseudovarieties.
 The reader may wish to consult~\cite{Rhodes&Steinberg:2009qt} or~\cite{Almeida:2003cshort} for background
 on pseudovarieties of monoids (or semigroups) and their relatively free objects.
 A \emph{pseudovariety of monoids $\pv V$} is a class of finite monoids closed under taking
 submonoids, homomorphic images, and finite direct products. If $\pv H$ is a pseudovariety of finite groups, then
 the class $\pv{\overline{H}}$ of finite monoids  whose subgroups belong to $\pv H$ is a pseudovariety of monoids.
 A \emph{pro-$\pv V$ monoid} is an inverse limit of monoids from $\pv V$, in the category of compact monoids.
 The class of $A$-generated compact monoids has free objects (as we have done in the rest of the paper, we assume that $A$ is finite). We let $\FV VA$ be the \emph{free pro-$\pv V$ monoid} generated by $A$.
 Note that for the pseudovariety  $\pv M$ of all finite monoids, we have $\FV MA=\F A$.

 Let $\pv H$ be a pseudovariety of finite groups that is extension-closed (i.e., closed under taking semidirect products).
 Roughly speaking, from the viewpoint adopted in this paper, $\FV {\overline{H}}A$ behaves pretty much like $\F A$ does.\footnote{When checking references in the literature, the reader should bear in mind that
   $\overline{\pv H}$ is often viewed as a \emph{pseudovariety of semigroups}, namely the pseudovariety of finite semigroups whose subgroups belong to $\pv H$. The two perspectives are essentially the same, for the purposes of this paper. As a pseudovariety of finite semigroups, $\overline{\pv H}$
   is a \emph{monoidal} pseudovariety of monoids, that is, it is generated by the monoids in~$\overline{\pv H}$.
   In particular, the free profinite \emph{semigroup} over the pseudovariety of \emph{semigroups}~$\overline{\pv H}$
   is the free profinite \emph{monoid} over the pseudovariety of \emph{monoids}~$\overline{\pv H}$
   minus the empty word, just as we have the equality $\FS A=\F A\setminus\{\varepsilon\}$.
 }
 For example, the free semigroup $A^*$
 embeds densely in~$\FV {\overline{H}}A$ and the elements of~$A^*$
 are isolated in~$\FV {\overline{H}}A$.
 A version of Theorem~\ref{t:rational-open} holds, where we replace ``\emph{rational languages $L\subseteq A^*$}\,''
 by ``\emph{languages $L\subseteq A^*$ recognized by monoids in $\pv{\overline{H}}$}\,''.
 The hypothesis that $\pv H$ is extension-closed is needed to guarantee that
 the analog of Theorem~\ref{t:injective-extension}, which is Corollary~2.2 of the paper~\cite{Margolis&Sapir&Weil:1995},
 holds for codes $X\subseteq A^+$ such that $X^*$ is recognized by a monoid in~$\pv{\overline{H}}$.
  Also, roughly speaking, we may say that all the ``combinatorics on pseudowords'' used so far in this paper also works in $\FV {\overline{H}}A$: we have a version of Proposition~\ref{p:cut-common-suffix-prefix} (details may be found in~\cite[Section 2]{ACosta&Steinberg:2021}),
   and also a version of Lemma~\ref{l:refine-factorization} (cf.~\cite[Section 3]{Almeida&ACosta&Costa&Zeitoun:2019}).

   If $\Cl S$ is an irreducible subshift of $A^\ZZ$, then we may still consider an analog of $G(\Cl S)$ in $\FV {\overline{H}}A$,
   denoted $G_{\overline{\pv H}}(\Cl S)$, which is actually the image of $G(\Cl S)$
   in $\FV {\overline{H}}A$ under the canonical projection $\F A\to \FV {\overline{H}}A$.
   It turns out that $G_{\overline{\pv H}}(\Cl S)$ is also invariant under conjugacy, a fact observed in~\cite{ACosta:2006}, and crucially
   used to show the main result of~\cite{ACosta&Steinberg:2011}, which states that
   $G_{\overline{\pv H}}(\Cl S)$ is a free pro-$\pv H$ group of rank $\aleph_0$ if $\Cl S$ is nonperiodic, $\Block(\Cl S)$ is recognized by a monoid from $\pv {\overline{H}}$, and $\pv H$ is an extension-closed pseudovariety of groups containing $\ZZ/{p\ZZ}$ for infinitely many primes~$p$. 
   In fact, $G_{\overline{\pv H}}(\Cl S)$ is even invariant under flow equivalence, for every irreducible subshift~$\Cl S$~\cite{ACosta&Steinberg:2021}.
   
   With minimal adaptations, the proof of Theorem~\ref{t:invariance-under-eventual-conjugacy} yields the following generalization.   

   \begin{Thm}\label{t:relatively-free-invariance-under-eventual-conjugacy}
     Let $\pv H$ be an extension-closed pseudovariety of finite groups such that
     $\pv H$ contains $\ZZ/{p\ZZ}$ for infinitely many primes $p$.     
  If $\Cl S_1$ and $\Cl S_2$ are eventually conjugate irreducible subshifts, then $G_{\overline{\pv H}}(\Cl S_1)$ and $G_{\overline{\pv H}}(\Cl S_2)$ are
   isomorphic profinite groups.
 \end{Thm}

 \begin{proof}
     Since the language $(A^n)^*$ is recognized by $\ZZ/{n\ZZ}$ for every positive integer~$n$, our proof that $G(\Cl S)$ is invariant under eventual conjugacy also works
   in this more general setting. It suffices to pay attention to the following adaptations.
   \begin{itemize}
   \item We replace the natural projection $\ell^A\colon\F A\to \widehat{\ZZ}$,
     by the natural projection $\ell^{A,\pv {\overline{H}}}\colon \FV {\overline{H}}A\to \widehat{\ZZ}_{\pv H}$
     where $\widehat{\ZZ}_{\pv H}$ is the inverse limit of cyclic groups from $\pv H$;
   \item For each $n\in\NN$ such that $\ZZ/{n\ZZ}\in\pv H$, we replace the natural projection $\ell_n^A\colon\F A\to \ZZ/{n\ZZ}$,
     by the natural projection $\ell_n^{A,\pv {\overline{H}}}\colon \FV {\overline{H}}A\to \ZZ/{n\ZZ}$;
   \item We use only positive integers $n$ such that $\ZZ/{n\ZZ}\in\pv H$. In particular, the set $Q_i$ used in the proof of Lemma~\ref{l:qi-finite} is now the set of
   all primes $p$ such that $\ZZ/{p\ZZ}\in\pv H$ and $\ell_p^{A,\pv {\overline{H}}}(H_i)=0$. This requirement explains why we need
   that   $\pv H$ contains $\ZZ/{p\ZZ}$ for infinitely many primes $p$: it is to ensure that
   the proof of Theorem~\ref{t:invariance-under-eventual-conjugacy} works in the case where $Q_1$ and $Q_2$ are finite, more precisely to guarantee that
   if $Q_1$ and $Q_2$ are finite, then there are distinct prime numbers $p,q$  not belonging to $Q_1\cup Q_2$ and greater than the threshold~$N$ such that
   both $\ZZ/{p\ZZ}$ and $\ZZ/{q\ZZ}$ belong to $\pv H$.
 \end{itemize}
 All arguments adapted from the proof of Theorem~\ref{t:invariance-under-eventual-conjugacy} rely on the aforementioned generalizations/adaptations of Theorem~\ref{t:rational-open}, Proposition~\ref{p:cut-common-suffix-prefix},
 Lemma~\ref{l:refine-factorization} and Theorem~\ref{t:injective-extension}, which entail transparent generalizations of the results of the previous sections on which
 Theorem~\ref{t:invariance-under-eventual-conjugacy} depends.
\end{proof}

\bibliographystyle{amsalpha}

\newcommand{\etalchar}[1]{$^{#1}$}
\providecommand{\bysame}{\leavevmode\hbox to3em{\hrulefill}\thinspace}
\providecommand{\MR}{\relax\ifhmode\unskip\space\fi MR }
\providecommand{\MRhref}[2]{%
  \href{http://www.ams.org/mathscinet-getitem?mr=#1}{#2}
}
\providecommand{\href}[2]{#2}

\end{document}